\documentclass[12pt]{article}
\usepackage[latin1]{inputenc}

\usepackage{enumitem}

\usepackage{amsmath,amsthm,amssymb,amsfonts}
\usepackage{mathtools}
\usepackage[pdftex]{graphicx}
\usepackage{graphicx}
\usepackage{epsfig}
\usepackage{verbatim}
\usepackage{footmisc}
\usepackage{bbm,bm}
\usepackage{lscape}
\usepackage[longnamesfirst]{natbib}
\usepackage{color}
\usepackage{float}
\usepackage{longtable}
\usepackage{lipsum}
\usepackage{subfigure}

\usepackage{booktabs}

\usepackage{natbib}

\makeatletter

\@addtoreset{equation}{section}
\makeatother

\allowdisplaybreaks 

\renewcommand{\baselinestretch}{1.2}

\newcommand{\ns}{\lfloor ns\rfloor}
\newcommand{\ep}{\varepsilon}
\newcommand{\Y}{\mathcal{Y}}

\newcommand{\lf}{\left\lfloor}
\newcommand{\rf}{\right\rfloor}
\newcommand{\R}{\mathbb{R}}
\newcommand{\N}{\mathbb{N}}
\newcommand{\Z}{\mathbb{Z}}
\def\cF{\mathcal{F}}
\newcommand{\om}{\omega}
\newcommand{\si}{\sigma}

\def\en{\mathbb{N}}

\def\er{\mathbb{R}}

\def\cf{\mathcal{F}}

\def\e{\varepsilon}

\def\be{\bar{\epsilon}_n}
\def\sjn{\sum_{j=1}^n}

\def\beq{\begin{eqnarray*}}
\def\eeq{\end{eqnarray*}}

\def\ns{\lfloor ns\rfloor}

\def\sins{\sum_{i=1}^{\lfloor ns\rfloor}}
\def\sinsk{\sum_{i=L+1}^{L+\lfloor\kappa_ns\rfloor}}

\newcommand{\nto}{\xrightarrow[n\to\infty]{}}
\newcommand{\bx}{{\bm{x}}}
\newcommand{\bX}{{\bm{X}}}
\newcommand{\bz}{{\bm{z}}}
\newcommand{\bj}{{\bm{j}}}

\def\nn{\nonumber}

\newtheorem{theo}{Theorem}[section]
\newtheorem{lemma}[theo]{Lemma}

\newtheorem*{Rem}{Remark}

\newtheorem*{Example}{Example}

\usepackage{hyperref}

\begin{document}

\title{\bf Estimating change points in nonparametric time series regression models}

\author{{\sc Maria Mohr} and {\sc Leonie Selk\footnote{ Financial support by the DFG (Research Unit FOR 1735 {\it Structural Inference in
Statistics: Adaptation and Efficiency}) is gratefully acknowledged.}}\\ Department of Mathematics, University of Hamburg}

\maketitle

\renewcommand{\baselinestretch}{1.1}

\begin{abstract}
In this paper we consider a regression model that allows for time series covariates as well as heteroscedasticity with a regression function that is modelled nonparametrically. We assume that the regression function changes at some unknown time $\lfloor ns_0\rfloor$, $s_0\in(0,1)$, and our aim is to estimate the (rescaled) change point $s_0$. The considered estimator is based on a Kolmogorov-Smirnov functional of the marked empirical process of residuals. We show consistency of the estimator and prove a rate of convergence of $O_P(n^{-1})$ which in this case is clearly optimal as there are only $n$ points in the sequence. Additionally we investigate the case of lagged dependent covariates, that is, autoregression models with a change in the nonparametric (auto-) regression function and give a consistency result. The method of proof also allows for different kinds of functionals such that Cram\'er-von Mises type estimators can be considered similarly. The approach extends existing literature by allowing nonparametric models, time series data as well as heteroscedasticity. Finite sample simulations indicate the good performance of our estimator in regression as well as autoregression models and a real data example shows its applicability in practise. 
\end{abstract}

\vspace*{.5cm}




\small
\normalsize

\section{Introduction}

Change point analysis has gained attention for decades in mathematical statistics. There is a vast literature on testing for structural breaks when the possible timing of such a break, the \textit{change point}, is unknown, see for instance \cite{Kirch2012365} and reference mentioned therein. This paper, however, is concerned with the estimation of the change point when assuming its existence. 

The most simple set of models can be described as follows
\[Y_t=\mu_1I\{t\le\lf ns_0\rf\}+\mu_2I\{t>\lf ns_0\rf\}+\ep_t, \ t=1,\dots, n,\]
where $s_0\in(0,1)$ is the (rescaled) change point, $\mu_1$ and $\mu_2$ the signal before and after the break, respectively, and $(\ep_t)_t$ being stationary and centred errors. These models are often referred to as AMOC-models (at most one change). The problem naturally moved from the standard case with independent errors (see \cite{Ferger1992345} among others) to the time series context. Both \cite{Bai1994435} and \cite{Antoch1997291} allow for linear processes and \cite{Huskova2008947} more generally for dependent errors. 

Additional information on the form of the signal can be expressed through a process of covariates $(X_t)_t$ resulting in linear regression models with a change in the regression parameter, such as
\[Y_t=\beta_1X_tI\{t\le\lf ns_0\rf\}+\beta_2X_tI\{t>\lf ns_0\rf\}+\ep_t, \ t=1,\dots, n,\]
where $\beta_1$ and $\beta_2$ are the regression coefficients before and after the break, respectively. \cite{Bai1997551}, \cite{Horvat1997109}, \cite{Aue2012367} among others consider the estimation of a change point in (multiple) linear regression models making use of least squares estimation. Considering $X_t=Y_{t-1}$ in the linear regression model from above, one obtains autoregressive models with one change in the autoregressive parameter. The estimation of the parameters and the unknown change point in AR(1) models was for instance considered by \cite{Chong200187}, \cite{Pang2014133} and \cite{Pang20154848}.

Our aim is to propose an estimator for the change point $s_0$ in a nonparametric version of the regression model from above, namely
\[Y_t=m_{(1)}(X_t)I\{t\le\lf ns_0\rf\}+m_{(2)}(X_t)I\{t>\lf ns_0\rf\}+\ep_t, \ t=1,\dots, n,\]
for some nonparametric regression functions $m_{(1)}, m_{(2)}$ (before and after the break) and in addition also investigate the autoregressive case where $X_t=Y_{t-1}$. While the investigation of points of discontinuity in (nonparametric) regression functions has been studied to some extend (see for instance \cite{Doering2015595} for an overview), not that much research has been devoted to change point analysis in nonparametric models as the one above, where the change occurs in time. \cite{Delgado2000113} propose estimators for the location and size of structural breaks in a nonparametric regression model imposing scalar breaks in time or values taken by some regressors, as in threshold models. Their rates of convergence and limiting distribution depends on a bandwidth, chosen for the kernel estimation. \cite{Chen200579} estimate the time of a scalar change in the conditional variance function in nonparametric heteroscedastic regression models using a hybrid procedure that combines the least squares and nonparametric methods.

The paper at hand extends existing literature, on the one hand by allowing for nonparametric heteroscedastic regression models with a general change in the unknown regression function where both errors and covariates are allowed to be time series, and on the other hand by investigating the autoregressive case.
The achieved rate of convergence for the proposed estimator of $O_P(n^{-1})$ is optimal as described in \cite{Hariz20071802}.

The remainder of the paper is organized as follows. The model and the considered estimator are introduced in section 2. Section 3 contains the regularity assumptions as well as the asymptotic results for the proposed estimator. Section 4 is concerned with the special case of lagged dependent covariates, that is the autoregressive case. In section 5 we describe a simulation study and discuss a real data example, whereas section 6 concludes the paper. Proofs of the main results as well as auxiliary lemmata can be found in the appendix.

\section{The model and estimator}\label{sec:model}

Let $\{(Y_t,\bm{X}_t):t\in\N\}$ be a weakly dependent stochastic process in $\R\times\R^d$ following the regression model
\begin{equation} \label{model} Y_{t}=m_{t}(\bX_{t})+U_{t}, \ t\in\N.\end{equation}

The unobservable innovations are assumed to fulfill $E[U_t|\cF^t]=0$ almost surely for the sigma-field $\cF^t=\si(U_{j-1},\bX_j:j\le t)$. We assume there exists a change point in the regression function such that
\begin{equation}\label{cpmodel}
m_{n,t}(\cdot)=m_{t}(\cdot)=\begin{cases}m_{(1)}(\cdot),\quad & t=1,\ldots,\lfloor ns_0\rfloor\\
m_{(2)}(\cdot), & t=\lfloor ns_0\rfloor +1,\ldots, n\end{cases},\qquad m_{(1)}\not\equiv m_{(2)}
\end{equation}
where $\lfloor ns_0\rfloor$ with $s_0\in(0,1)$ is the unknown time the change occurs. Note that we keep above notations for simplicity reasons, however, the considered process is in fact a triangular array process $\{(Y_{n,t},\bm{X}_{n,t}):1\le t\le n,n\in\N\}$ and will be treated appropriately. 

Assuming $(Y_1,\bX_1),\dots,(Y_n,\bX_n)$ have been observed, the aim is to estimate $s_0$. The idea is to base the estimator on the \textit{sequential marked empirical process of residuals}, namely
\[\hat{T}_n(s,\bm{z}):=\frac 1n\sins (Y_{i}-\hat m_n(\bX_{i}))\omega_n(\bX_{i})I\{\bX_{i}\leq \bz\},\]
for $s\in[0,1]$ and $\bm{z}\in\R^d$, where $\bm{x}\le \bm{y}$ is short for $x_j\le y_j$ for all $j=1,\dots, d$, $\om_n(\cdot)=I\{\cdot \in\bm{J}_n\}$ being from assumption (J) below and $\hat{m}_n$ being the Nadaraya-Watson estimator, that is
\[\hat{m}_n(\bm{x})=\frac{\sum_{j=1}^{n}K\left(\frac{\bm{x}-\bm{X}_j}{h_n}\right)Y_j}{\sum_{j=1}^{n}K\left(\frac{\bm{x}-\bm{X}_j}{h_n}\right)},\]
with kernel function $K$ and bandwidth $h_n$ as considered in the assumptions below. Then we want to estimate $s_0$ by
\begin{equation}\label{eq:def estimator}
\hat s_n:=\min\left\{s:\sup_{\bz\in\er^d}|\hat{T}_n(s,\bm{z})|=\sup_{s\in[0,1]}\sup_{\bz\in\er^d}|\hat{T}_n(s,\bm{z})|\right\}.
\end{equation}
Note that $\hat s_n=\lf n \hat s_n\rf/n$.

\begin{Rem}
The advantage of using marked residuals in comparison to using the classical CUSUM $\hat T_n(s,\bm{\infty})$ to estimate the change point is that in the first case the estimator is consistent for all changes of the form \eqref{cpmodel} whereas there are several examples in which the use of $\hat T_n(s,\bm{\infty})$ leads to a non-consistent estimator.
To this end see the remark below the proof of Theorem \ref{cons} and compare to \cite{Mohr2019}.
\end{Rem}

\begin{Rem}
\cite{Mohr2019} constructed procedures based on functionals of $\hat{T}_n$, e.g.~a Kolmogorov-Smirnov test statistic $\sup_{s\in[0,1]}\sup_{\bm{z}\in\R^d}|\hat{T}_n(s,\bm{z})|$, to test the null hypothesis of no changes in the unknown regression function against change point alternatives as in \eqref{cpmodel}. Given that such a test has rejected the null, the use of an M-estimator as in \eqref{eq:def estimator} seems natural. Furthermore, Cram\'er-von Mises type test statistics of the form $\sup_{s\in[0,1]}\int_{\R^d}|\hat{T}_n(s,\bm{z})|^2\nu (\bm{z})d\bm{z}$ for some integrable $\nu:\R^d\to\R$ were also considered by \cite{Mohr2019}. Assuming strict stationarity of the covariates and the existence of a density $f$ such that $\bX_t\sim f$ for all $t$, as in (X1) below, the Cram\'er-von Mises approach from above with $\nu\equiv f$ leads to an alternative estimator for $s_0$, namely
\[\tilde{s}_n:=\min\left\{s:\left(\int_{\R^d}|\hat{T}_n(s,\bm{z})|^2f (\bm{z})d\bm{z}\right)^{1/2}=\sup_{s\in[0,1]}\left(\int_{\R^d}|\hat{T}_n(s,\bm{z})|^2f (\bm{z})d\bm{z}\right)^{1/2}\right\}.\]
However, to obtain a feasible estimator one needs to replace the integral $\int_{\R^d}|\hat{T}_n(s,\bm{z})|^2f (\bm{z})d\bm{z}$ by its empirical counterpart $\frac{1}{n}\sum_{k=1}^{n}|\hat{T}_n(s,\bm{X}_k)|^2$ in practise as $f$ is not known. 

%
%
\end{Rem}

\section{Asymptotic results}  \label{asymp}

In this section we will derive asymptotic properties for $\hat{s}_n$. To this end we introduce the following assumptions.

\begin{itemize}
\item[(U)] For all $t\in\Z$ let $E[U_t|\cF^t]=0$ a.s. for $\cF^t=\si(U_{j-1},\bX_j:j\le t)$ and $E[|U_{t}|^q]\le C_{U}$ for some $C_U<\infty$ and $q> 2$.
\item[(M)] For all $t\in\Z$ let $E[|m_{(1)}(\bm{X}_t)-m_{(2)}(\bm{X}_t)|^r]\le C_{m}$ for some $C_m<\infty$ and $r> 2$.
\item[(P)] Let $\{(Y_t,\bX_t):1\le t\le n,n\in\N\}$ be strongly mixing with mixing coefficient $\alpha(\cdot)$. For $q,r$ from assumptions (U) and (M) and $b:=\min(q,r)$ let $\alpha(t)=O(t^{-\bar{\alpha}})$ with some $\bar{\alpha}>(1+(b-1)(1+d))/(b-2)$.
\item[(N)] For $b$ from assumption (P) let $E[|Y_t|^b]<\infty$ and let $\bm{X}_t$ be absolutely continuous with density function $f_t:\R^d\to\R$ that satisfies $\sup_{\bm{x}\in\R^d}E[|Y_t|^b|\bm{X}_t=\bm{x}]f_t(\bm{x})<\infty$ and $\sup_{\bm{x}\in\R^d}f_t(\bm{x})<\infty$ for all $t\in\{1,\dots,n\}$ and $n\in\N$. Let there exist some $N\ge 0$ such that  $\sup_{|i-j|\ge N}\sup_{\bm{x}_i,\bm{x}_j}E[|Y_iY_j||\bm{X}_i=\bm{x}_i,\bm{X}_j=\bm{x}_j]f_{ij}(\bm{x}_i,\bm{x}_j)<\infty$ for all $n\in\N$, where $f_{ij}$ is the density function of $(\bm{X}_i,\bm{X}_j)$.
\item[(J)] Let $(c_n)_{n\in\N}$ be a positive sequence of real valued numbers satisfying $c_n\to \infty$ and $c_n=O((\log{n})^{1/d})$ and let $\bm{J}_n=[-c_n,c_n]^d$.
\item[(F)] For some $C<\infty$ and $c_n$ from assumption (J) let $\bm{I}_n=[-c_n-Ch_n,c_n+Ch_n]^d$ and let $\delta_n^{-1}=\inf_{\bm{x}\in \bm{J}_n}\inf_{1\le t\le n}f_t(\bm{x})>0$ for all $n\in\N$. Further, let for all $n\in\N$
\begin{eqnarray*}
p_n&=&\max\limits_{|\bm{k}|=1}\sup\limits_{\bm{x}\in \bm{I}_n}\sup\limits_{1\le t\le n}|D^{\bm{k}}f_t(\bm{x})|<\infty\\
0<q_n&=&\max\limits_{0\le |\bm{k}|\le 1}\sup\limits_{\bm{x}\in \bm{I}_n}\max_{j=1,2}|D^{\bm{k}}m_{(j)}(\bm{x})|<\infty,
\end{eqnarray*}
where $|\bm{i}|=\sum_{j=1}^{d}i_j$ and $D^{\bm{i}}=\frac{\partial^{|\bm{i}|}}{\partial x_1^{i_1}\dots\partial x_d^{i_d}}$ for $\bm{i}=(i_1,\dots,i_d)\in\N_0^d$.
\item[(K)] Let $K:\R^d\to\R$ be symmetric in each component with $\int_{\R^d}K(\bm{z})d\bm{z}=1$ and compact support $[-C,C]^d$. Additionally let $|K(\bm{u})|<\infty$ for all $\bm{u}\in\R^d$ and $|K(\bm{u})-K(\bm{u'})|\le \Lambda \|\bm{u}-\bm{u'}\|$ for some $\Lambda<\infty$ and for all $\bm{u},\bm{u'}\in\R^d$, where $\|\bm{x}\|=\max_{i=1,\ldots,d}|x_i|$.
\item[(B)] With $b$ and $\bar{\alpha}$ from assumption (P) let \[\frac{\log{(n)}}{n^{\theta}h_n^d}=o(1) \text{ for } \theta=\frac{\bar{\alpha}-1-d-\frac{1+\bar{\alpha}}{b-1}}{\bar{\alpha}+3-d-\frac{1+\bar{\alpha}}{b-1}}.\]
For $\delta_n,p_n,q_n$ from assumption (F) let 
\[\left(\sqrt{\frac{\log (n)}{nh_n^{d}}}+h_np_n\right)p_nq_n\delta_n=o(n^{-\zeta})\]
for some $\zeta>0$.
\item[(X1)] For all $1\le t\le n, n\in\N$ let $f_t(\cdot)= f(\cdot)$, for some density $f$.
\item[(X2)] For all $1\le t\le n, n\in\N$ let $f_t(\cdot)= f_{(1)}(\cdot)$ for all $t=1,\ldots,\lfloor ns_0\rfloor$ and $f_t(\cdot)= f_{(2)}(\cdot)$ for all $t=\lfloor ns_0\rfloor +1,\ldots, n$, for some densities $f_{(1)}$, $f_{(2)}$.
\end{itemize}

\begin{Rem}
The assumptions on the error terms and the mixing assumptions particularly allow for conditional heteroscedasticity. Assumptions (U), (M) and (P) are a trade off between the existence of moments and the rate of decay of the mixing coefficient. Assumptions (P), (N), (K) and the first part of (B) are reproduced from \cite{Kristensen2012420}. Together with (J) and (F), they are used to obtain uniform rates of convergence for $\hat{m}_n$ stated in Lemma \ref{mdachRate} in the appendix. In $(X1)$, we assume stationarity of the covariates for the whole observation period, while in the case of $(X2)$ we assume stationarity before and right after the change occurs. Nevertheless both assumptions rule out general autoregressive effects such as $\bX_t=(Y_{t-1},\dots,Y_{t-d})$. We will address this issue separately in section \ref{auto}.
\end{Rem}

\begin{theo}\label{cons}
Assume (U), (M), (P), (N), (J), (F), (K) and (B). Furthermore let either (X1) or (X2) hold. Then the change point estimator $\hat s_n$ is consistent, i.\,e.\ 
\[|\hat s_n-s_0|=o_P(1).\]
\end{theo}

\begin{theo}\label{rates}
Under the assumptions of Theorem \ref{cons} for the change point estimator $\hat s_n$ it holds that
\[|\hat s_n- s_0|=O_P(r_n^{-1}),\]
where $r_n=n$.
\end{theo}

The proofs of the theorems can be found in appendix \ref{sec:prooftheorem}. We state both theorems seperately since we need Theorem \ref{cons} to prove Theorem \ref{rates}.

\begin{Rem}
To obtain the rates of convergence we make use of the fact that $\hat{s}_n$ can be expressed using the sup norm on $ l^{\infty}(\R^d)$, i.e.~
\[N: l^{\infty}(\R^d)\to \R, \ g\mapsto N(g):=\sup_{\bm{z}\in\R^d}|g(\bm{z})|,\]
where $ l^{\infty}(\R^d)$ is the space of all uniformly bounded real valued functions on $\R^d$. Note that similarly $\tilde{s}_n$ can be expressed using the $L_2(P)$ norm, when $(\bX_t)_{t}$ is strictly stationary with marginal distribution $P$, namely
\[\tilde{N}: l^{\infty}(\R^d)\to \R, \ g\mapsto \tilde{N}(g):=\left(\int_{\R^d}|g(\bm{z})|^2f(\bm{z})d\bm{z}\right)^{1/2}.\]
Using $\tilde{N}(g)\le N(g)$ for all $g\in l^{\infty}(\R^d)$, corresponding results for $\tilde{s}_n$ as in Theorem \ref{cons} and Theorem \ref{rates} can be proven in a similar matter.
\end{Rem}

\section{The autoregressive case}  \label{auto}

In this section we will consider the case where the exogenous variables include finitely many lagged values of the endogenous variable, we will refer to this model as the \textit{autoregressive case}. We will focus on one dimensional covariates, however, the results do not depend on the dimension and can also be formulated for higher order autoregression models. Consider the nonparametric autoregression
\begin{align}Y_t=m_t(Y_{t-1})+U_t, \ t=1,\dots,n,\label{eq:model AR}\end{align}
with unobservable innovations $U_t$ and one change in the regression function occurring at some unknown time $\lf ns_0 \rf$ as in \eqref{cpmodel}.

Furthermore assume the following.

\begin{itemize}
\item[(X3)] For all $1\le t\le n, n\in\N$ let $X_t:=Y_{t-1}$ be absolutely continuous with density $f_{t}$. Let there exist densities $f_{(1)}$ and $f_{(2)}$ such that $f_{t}(\cdot)= f_{(1)}(\cdot)$ for all $t=1,\ldots,\lfloor ns_0\rfloor$ and
$R_{n}(x):=\frac{1}{n}\sum_{j=\lf ns_0\rf +1}^{n}f_j(x)-\frac{n-\lf ns_0\rf}{n}f_{(2)}(x)\to 0$ for all $x\in\R$ and $n\to\infty$.
\end{itemize}

\begin{Rem} 
Note that (X3) requires on the one hand strict stationarity up to the time of change $\lf ns_0\rf$. On the other hand the time series needs to reach its (new) stationary distribution fast enough after the change. This is a generalization of (X2) where we assumed stationarity both before and right after the change point, which can not be fulfilled in the model \eqref{eq:model AR}. A necessary condition then is that there exists a stationary solution of equation \eqref{eq:model AR} under both $m_{(1)}(\cdot)$ and $m_{(2)}(\cdot)$ as regression functions. 
\end{Rem}

\begin{Example}
Consider the AR(1)-model
\[Y_t=a_t\cdot Y_{t-1}+\e_t\]
with standard normally distributed innovations $(\e_t)_t$ and $a_t=a\in(-1,1)$ for $t\leq\lfloor ns_0\rfloor$, $a_t=b\in(-1,1)$ for $t>\lfloor ns_0\rfloor$, $a\neq b$. Then assumption (X3) is fulfilled. Note to this end that $X_t:=Y_{t-1}\sim N(0,1/(1-a^2))$ for $t\leq\lfloor ns_0\rfloor$. The distribution after the change point is given by $X_{\lfloor ns_0\rfloor +1+k}\sim N\big(0,b^{2(k+1)}/(1-a^2)+\sum_{i=0}^kb^{2i}\big)$ for all $k>0$. Thus with $\sigma_j^2:=b^{2(j-\lfloor ns_0\rfloor)}/(1-a^2)+\sum_{i=0}^{j-\lfloor ns_0\rfloor -1}b^{2i}$ by the mean value theorem it holds for some $\xi_j$ between $\sigma_j^2$ and $(1-b^2)^{-1}$ that
\beq
R_n(x)&=&\frac 1n\sum_{j=\lfloor ns_0\rfloor+1}^n\left(\frac 1{\sqrt{2\pi\sigma_j^2}}\exp\left(-\frac {x^2}{2\sigma_j^2}\right)-\frac 1{\sqrt{2\pi(1-b^2)^{-1}}}\exp\left(-\frac {x^2}{2(1-b^2)^{-1}}\right)\right)\\
&=&\frac 1n\sum_{j=\lfloor ns_0\rfloor+1}^n\left(\sigma_j^2-\frac 1{1-b^2}\right)\exp\left(-\frac {x^2}{2\xi_j}\right)\left(-\frac 12\cdot\frac1{(2\pi\xi_j)^{\frac 32}}+\frac1{(2\pi\xi_j)^{\frac 12}}\cdot\frac{x^2}{2\xi_j^2}\right)\\
&\leq&C\frac 1n\sum_{j=\lfloor ns_0\rfloor+1}^n\left|\sigma_j^2-\frac 1{1-b^2}\right|\\
\eeq
for some constant $C<\infty$ for all $x\in\er$. Further we can conclude
\beq
\frac 1n\sum_{j=\lfloor ns_0\rfloor+1}^n\left|\sigma_j^2-\frac 1{1-b^2}\right|&=&\frac 1n\sum_{j=\lfloor ns_0\rfloor+1}^n\left|\frac{b^{2(j-\lfloor ns_0\rfloor)}}{1-a^2}+\frac{1-b^{2(j-\lfloor ns_0\rfloor)}}{1-b^2}-\frac 1{1-b^2}\right|\\
&=&\left|\frac 1{1-a^2}-\frac 1{1-b^2}\right|\frac 1n\sum_{j=\lfloor ns_0\rfloor+1}^nb^{2(j-\lfloor ns_0\rfloor)}
\eeq
and thus $R_n(x)\nto 0$ for all $x\in\er$.
\end{Example}

In general verifying assumption (X3) for model \eqref{eq:model AR} means to compare the distribution of a stochastic process that is not yet in balance with its stationary distribution. A well known technique to deal with this task is coupling, see e.\,g.\ \cite{Franke2002555}. 

Under (X3) we get the following consistency result for our change point estimator in the autoregressive case.

%

\begin{theo}\label{consAR}
Assume model \eqref{eq:model AR} under (U), (M), (P), (N), (J), (F), (K), (B) and (X3). Then the change point estimator $\hat s_n$ is consistent, i.\,e.\ 
\[|\hat s_n-s_0|=o_P(1).\]
\end{theo}

The proof can be found in appendix \ref{sec:prooftheorem}.

\begin{Rem}
Another possibility to handle the autoregressive case would be to model the change in a different way, namely
\[Y_t=\begin{cases} Y^{(1)}_t=m_{(1)}\big(Y^{(1)}_{t-1}\big)+U^{(1)}_t,\quad & t=1,\ldots,\lfloor ns_0\rfloor\\Y^{(2)}_t=m_{(2)}\big(Y^{(2)}_{t-1}\big)+ U^{(2)}_t,\quad & t=\lfloor ns_0\rfloor+1,\ldots,n\end{cases},\qquad m_{(1)}\not\equiv m_{(2)},\]
for two stationary processes $\big(Y^{(1)}_t\big)_t$, $\big(Y^{(2)}_t\big)_t$, see e.\,g.\ \cite{Kirch20151197}. In this case assumption (X2) is fulfilled and thus Theorem \ref{cons} and Theorem \ref{rates} apply.
\end{Rem}

\section{Finite sample properties}  \label{simus}

\subsection{Simulations}

To investigate the finite sample performance of our estimator, we generate data from two different basic models, namely 
\begin{itemize}
\item[(IID)] $Y_t=m_t(X_t)+\sigma(X_t)\e_t$, where the observations $(X_t)_t$ are i.i.d., univariate and standard normally distributed, just as the errors $(\e_t)_t$.
\item[(TS)] $Y_t=m_t(X_t)+\sigma(X_t)\e_t$, where $(\e_t)_t$ i.i.d.~$\sim N(0,1)$ and the univariate observations $(X_t)_t$ stem from a time series $X_t=0.4X_{t-1}+\eta_t$ with standard normal innovations $(\eta_t)_t$.
\end{itemize}
For both models we generate data both for the homoscedastic case $\sigma\equiv 1$ as well as for the heteroscedastic case $\sigma(x)=\sqrt{1+0.5x^2}$. The results for both are very similar in all situations, thus we only present the results for the heteroscedastic case.
To model the change in the regression function we use three different scenarios
\begin{itemize}
\item[(C1)]$m_t=\begin{cases} -0.5x,\quad & t=1,\ldots,\lfloor ns_0\rfloor\\ 0.5 x\quad & t=\lfloor ns_0\rfloor +1,\ldots,n,\end{cases}$
\item[(C2)]$m_t=\begin{cases} 0.1x,\quad & t=1,\ldots,\lfloor ns_0\rfloor\\ 0.9 x\quad & t=\lfloor ns_0\rfloor +1,\ldots,n,\end{cases}$
\item[(C3)]$m_t=\begin{cases} 0.5x,\quad & t=1,\ldots,\lfloor ns_0\rfloor\\ (0.5+3\exp(-0.8x^2)) x\quad & t=\lfloor ns_0\rfloor +1,\ldots,n,\end{cases}$
\end{itemize}
where we let $s_0$ range from $0.1$ to $0.9$. In Figure \ref{simus1}  the results for 1000 replications and sample sizes $n=100,500,1000$ are shown, where we plot $s_0$ against the estimated mean squared error of our estimator $\hat s_n$. The kernel for $\hat m_n$ is chosen as the Epanechnikov kernel of order four and the bandwidth is determined by a cross-validation method. It can be seen that our estimator performs quite well even for the smallest sample size $n=100$ when $s_0$ is $0.5$ or close to it whereas for a change point  that lies closer to the boundaries of the observation interval a larger sample size is needed to get satisfying results. This is due to the fact that if $s_0=0.1$ or $s_0=0.9$ there are only $10$ observations before and after the change point respectively for $n=100$ and thus the estimation of $m_{(1)}$ and $m_{(2)}$ respectively are poor. Moreover an asymmetry in the results is striking. This stems from the CUSUM type statistic that our estimator is based on. For $s_0=0.1$ e.\,g.\ the sum consists of only $0.1n$ summands and thus the estimation of e.\,g.\ $E[U_t]$ is worse than if $s_0=0.9$ and the estimation is based on $0.9n$ summands.
The effect of a decreasing performance of the estimators the closer $s_0$ gets to the boundaries is typical for change point estimators based on CUSUM statistics and can be antagonized by the use of appropriate weights, see e.\,g.\ \cite{Ferger2005255}.

\begin{figure}[ht]
  \centering
  (IID)\hspace{7cm} (TS)\\
   \vspace{-0.25cm}
  \subfigure{\includegraphics[width=0.49\textwidth]{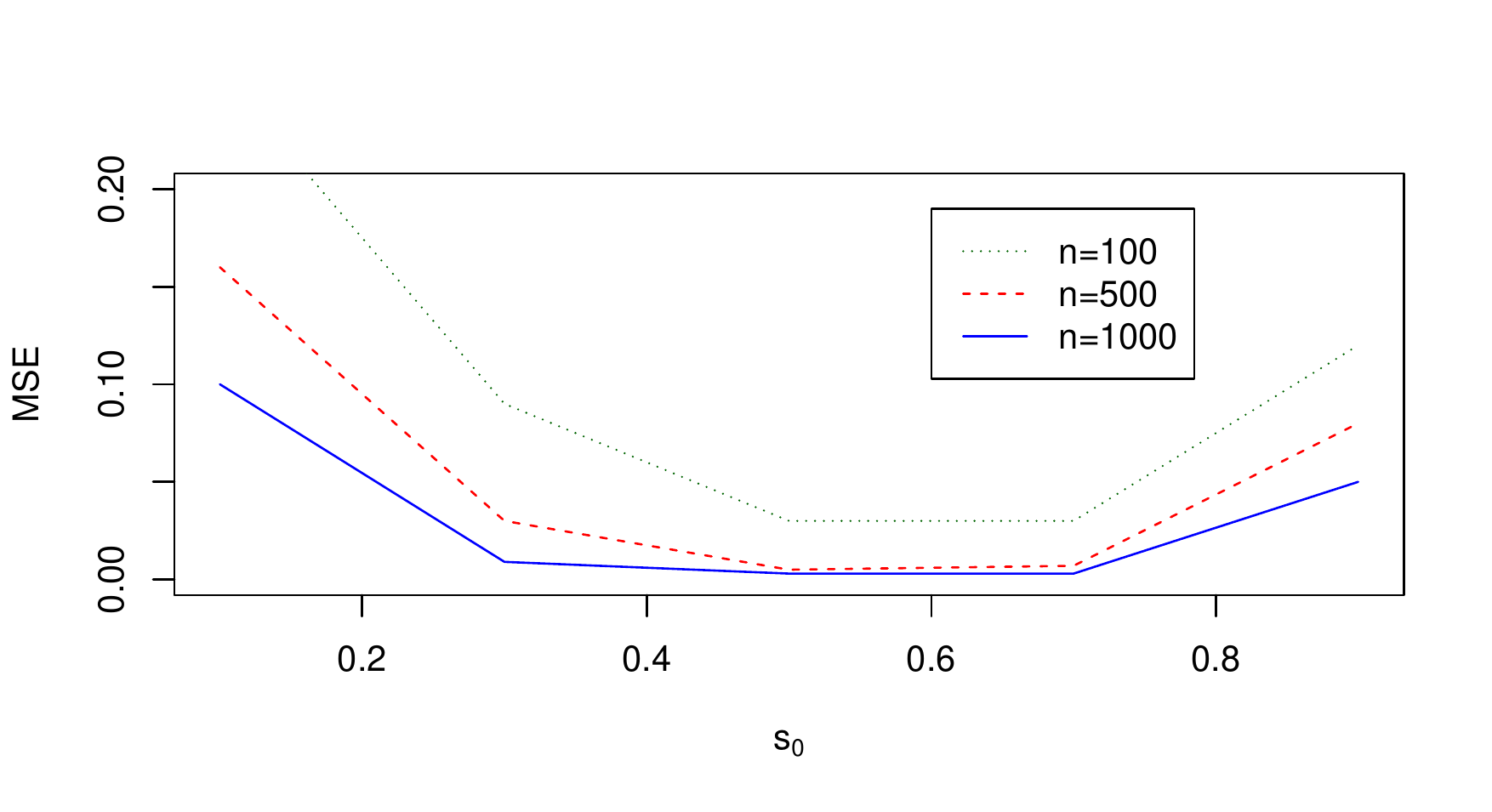}} 
    \subfigure{\includegraphics[width=0.49\textwidth]{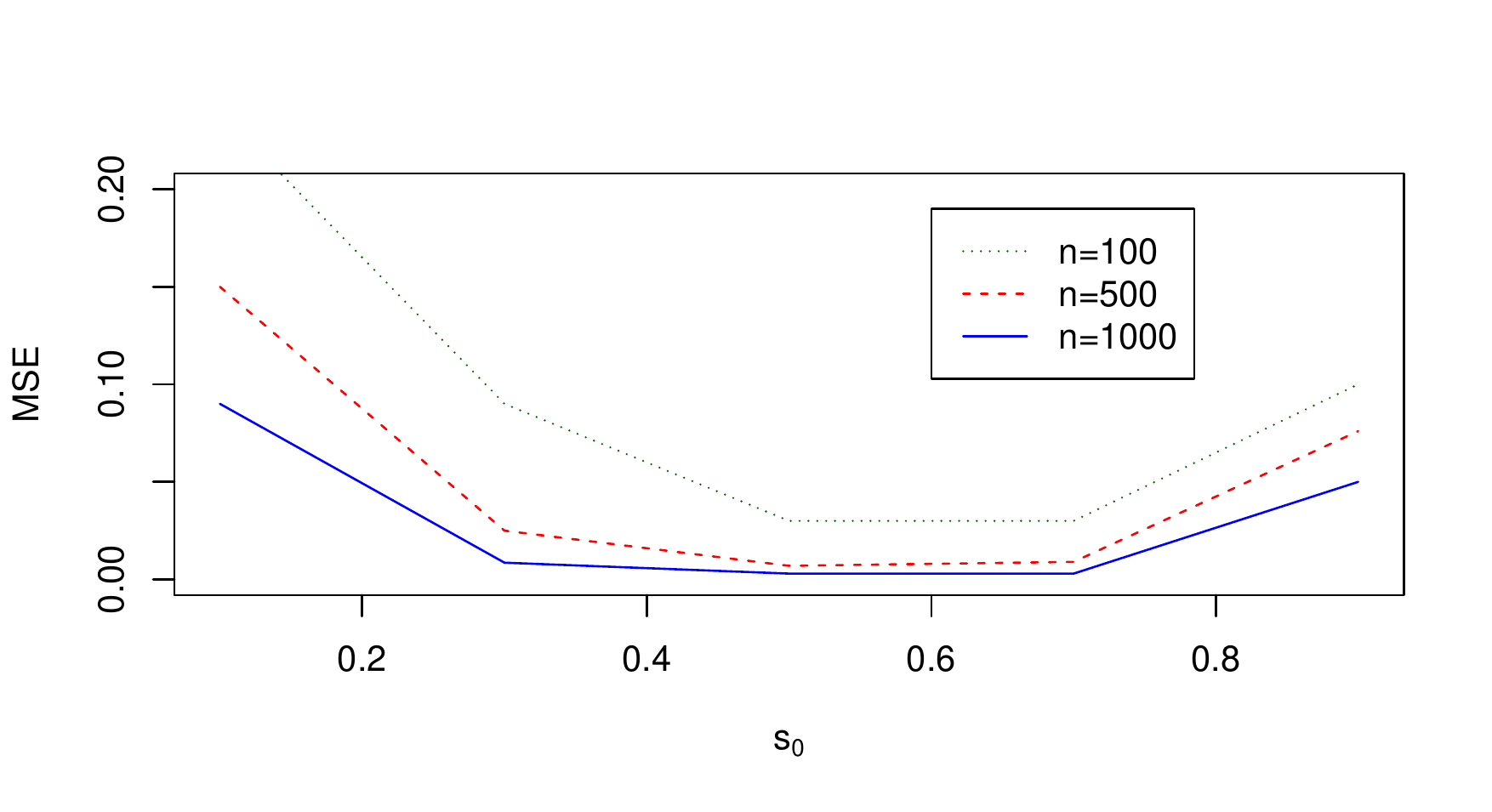}} 
   \\ \vspace{-0.7cm}
     \subfigure{\includegraphics[width=0.49\textwidth]{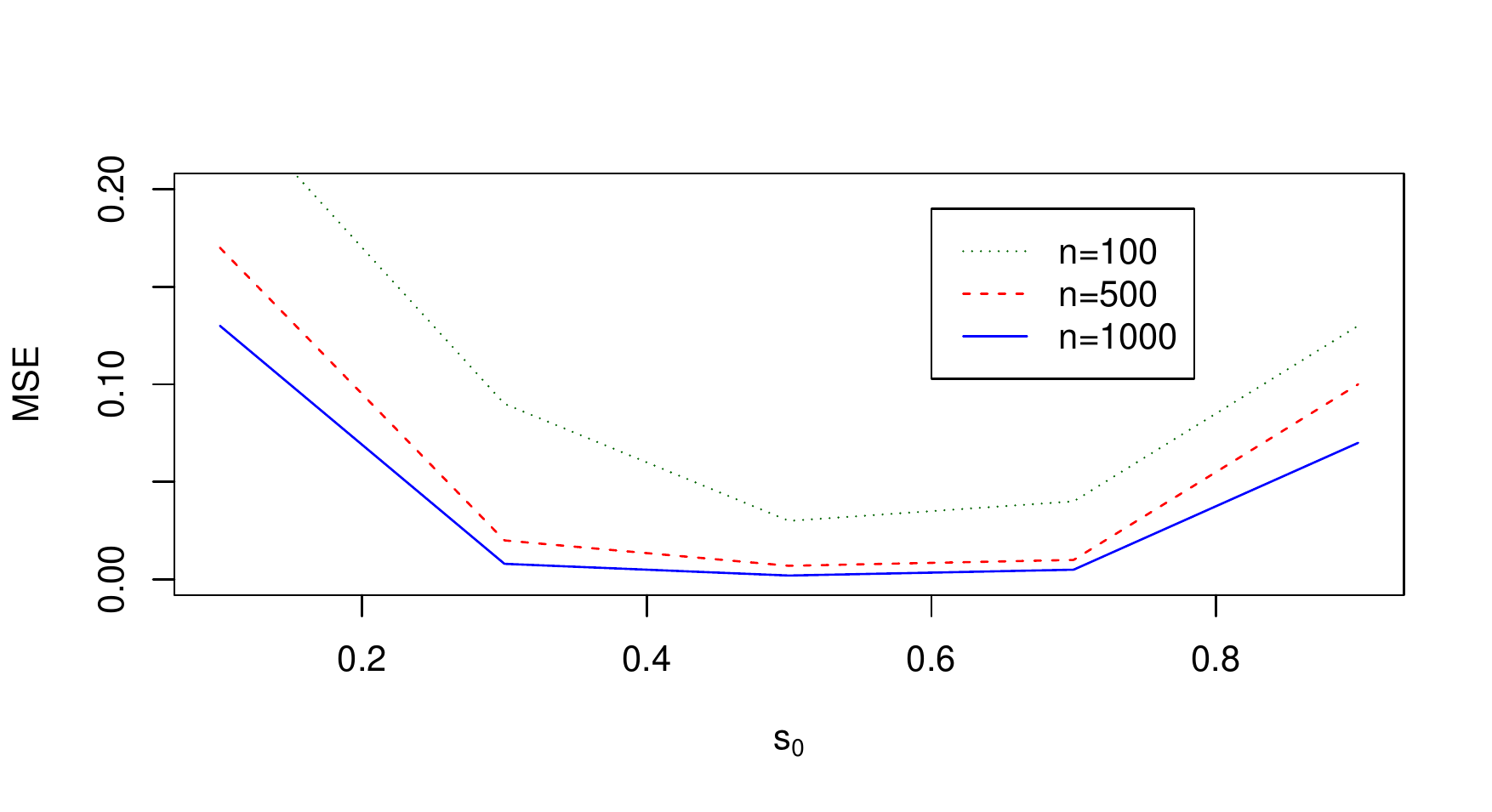}} 
    \subfigure{\includegraphics[width=0.49\textwidth]{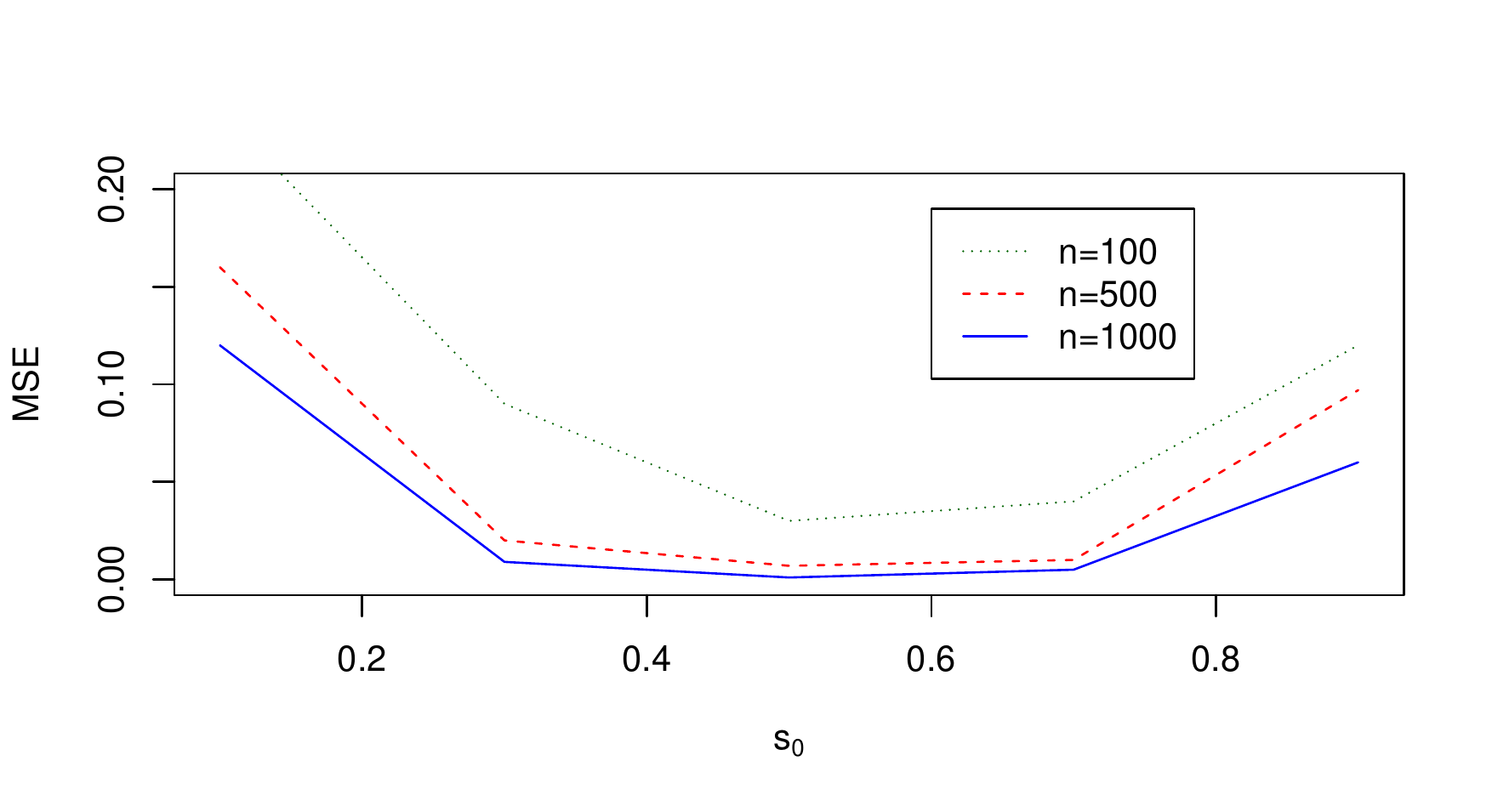}} 
    \\ \vspace{-0.7cm}
     \subfigure{\includegraphics[width=0.49\textwidth]{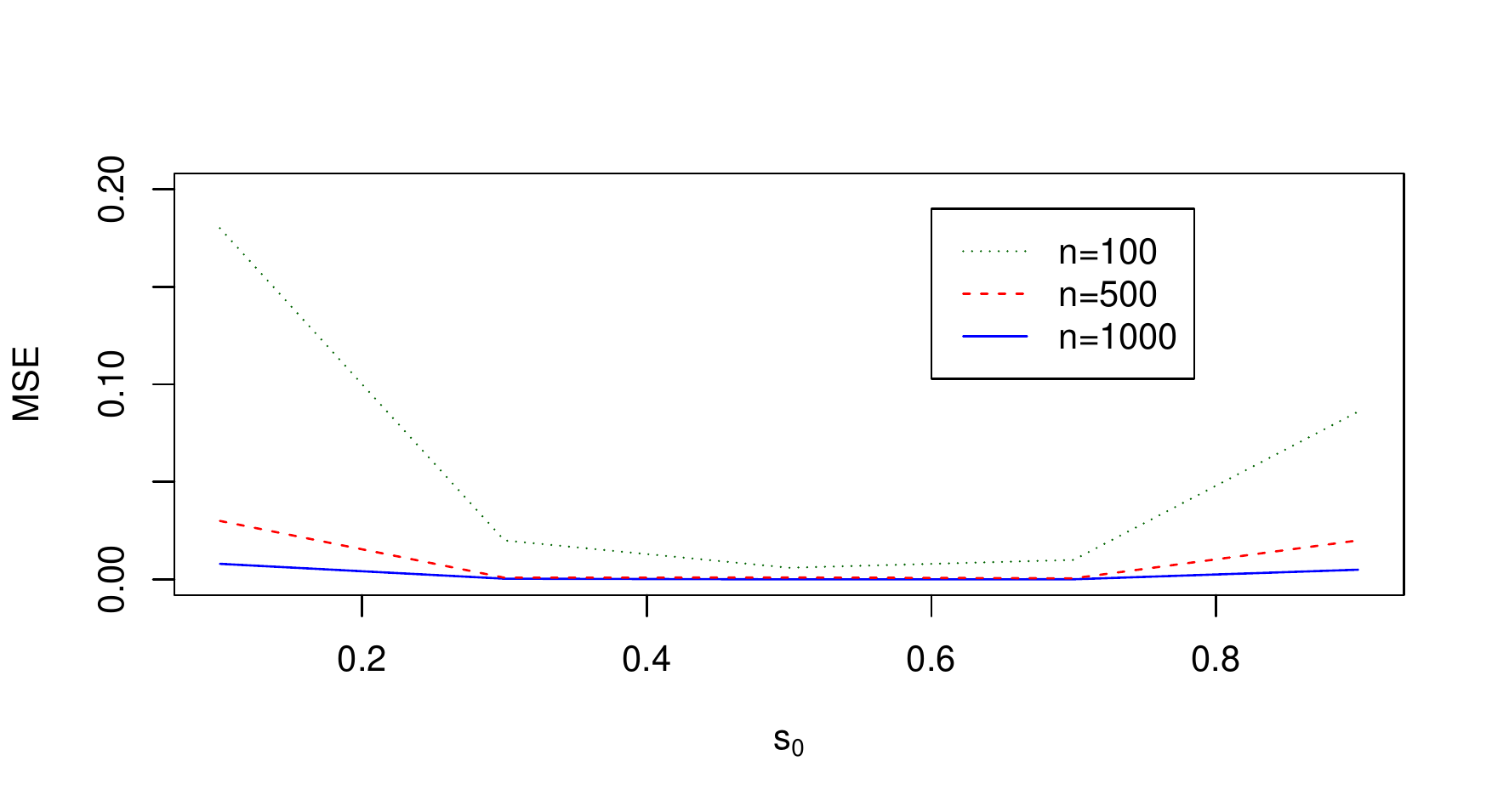}} 
    \subfigure{\includegraphics[width=0.49\textwidth]{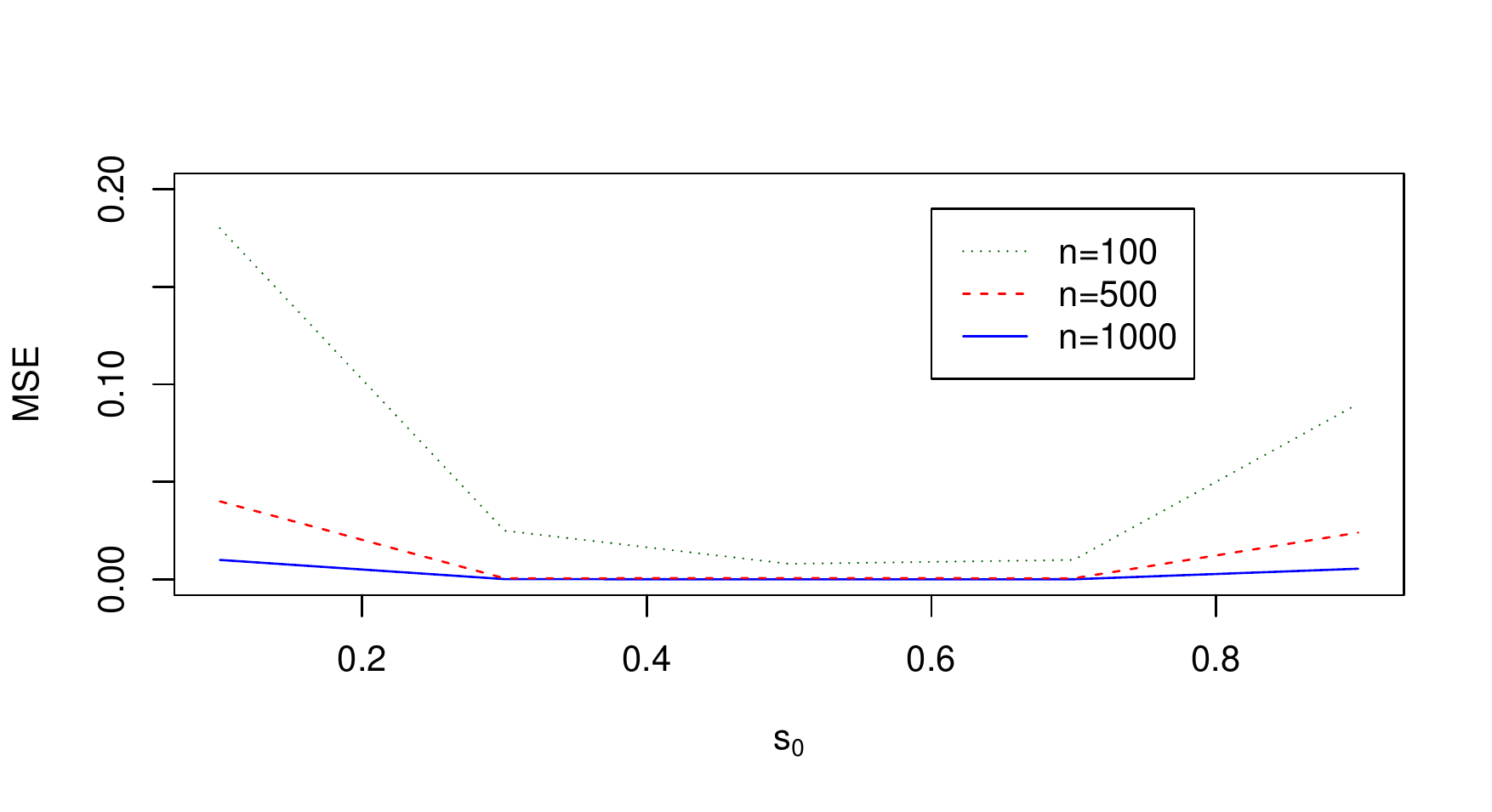}} 
  \caption{Simulation results for model (IID) (left), model (TS) (right) and change point scenario (C1) (top), change point scenario (C2) (middle), change point scenario (C3) (bottom)}
  \label{simus1}
\end{figure}

To stress our estimator a little further we simulate the scenario that there is also a change in the variance function $\sigma$ at a different time point than the change in the regression function $m$. In this situation the estimator should still be able to detect $s_0$, the change point in the regression function. The results are shown in Figure \ref{simus2} for model (IID) and model (TS) with change point scenario (C1) where $\sigma_t(x)=\sqrt{1+0.1x^2}$ for $t\leq 0.4n$ and $\sigma_t(x)=\sqrt{1+0.8x^2}$ for $t>0.4n$. They confirm the good performance of our estimator even in this more difficult situation.

\begin{figure}[ht]
  \centering
   (IID)\hspace{7cm} (TS)\\
     \vspace{-0.25cm}
    \subfigure{\includegraphics[width=0.49\textwidth]{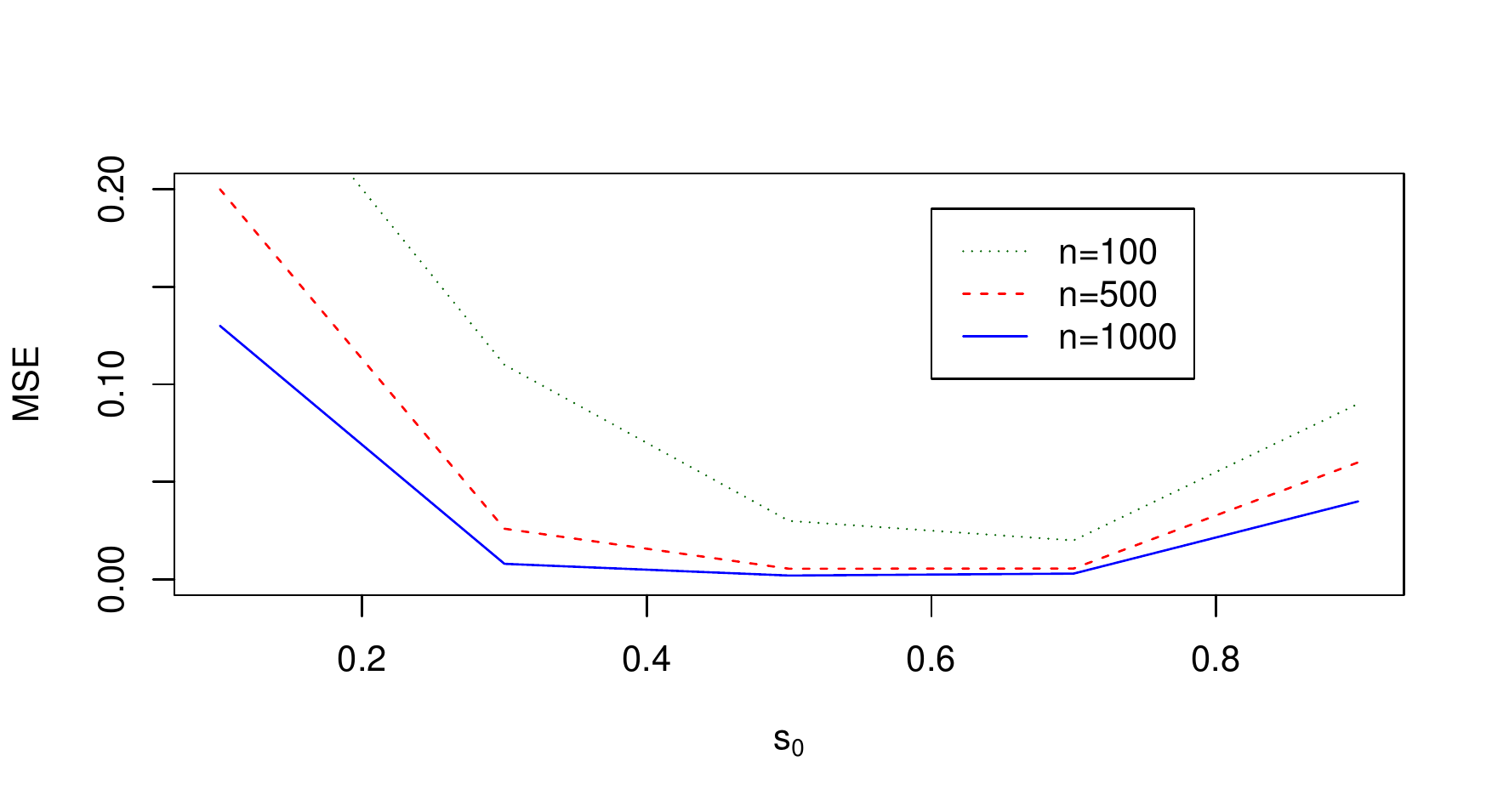}}
     \subfigure{\includegraphics[width=0.49\textwidth]{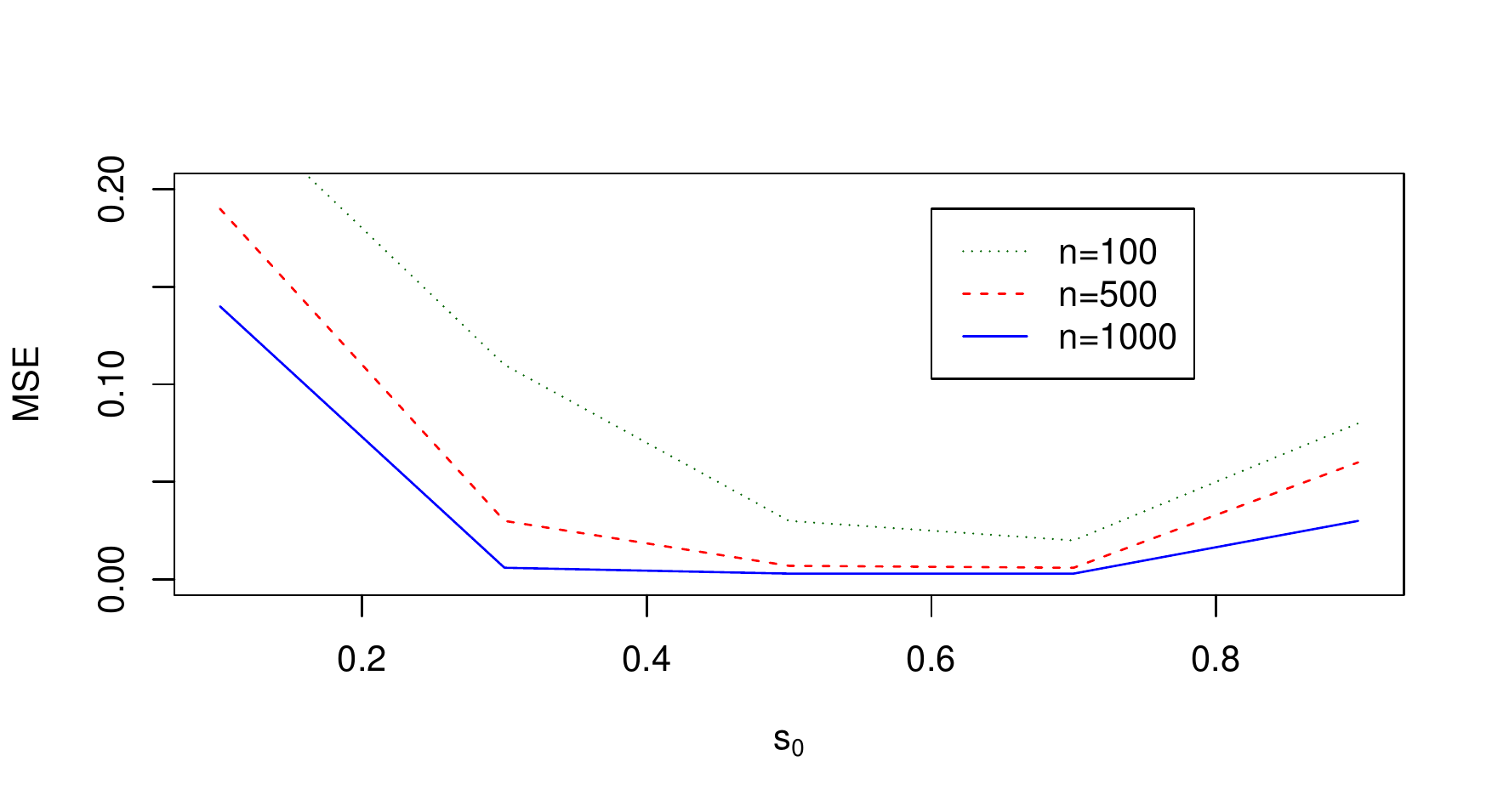}}
  \caption{Simulation results for model (IID) (left) and model (TS) (right) with change point scenario (C1) and an additional change in the variance function}
  \label{simus2}
\end{figure}

As discussed in section \ref{auto} our estimator can also be applied to the autoregressive case. To investigate the finite sample performance in this situation we generate data according to the model
\begin{itemize}
\item[(AR)] $Y_t=m_t(Y_{t-1})+\sigma(Y_{t-1})\e_t$, where $(\e_t)_t$ i.i.d.~$\sim N(0,1)$.
\end{itemize}
For $\sigma\equiv 1$ and change point scenario (C1) as well as (C2) assumption (X3) is fulfilled, see the example in section \ref{auto}. Simulation results for these cases are shown in Figure \ref{simus3}  where the setting is the same as described above. They look very similar to the results of model (IID) and (TS) and thus confirm the theoretical result of Theorem \ref{consAR}.
Even for examples where assumption (X3) can not be verified easily the performance of our estimator is satisfying, see Figure \ref{simus4} for model (AR) with $\sigma\equiv 1$ and change point scenario (C3) as well as the heteroscedastic model (AR) with $\sigma=\sqrt{1+0.5x^2}$ and change point scenario (C1). 

\begin{figure}[ht]
  \centering
  \subfigure{\includegraphics[width=0.49\textwidth]{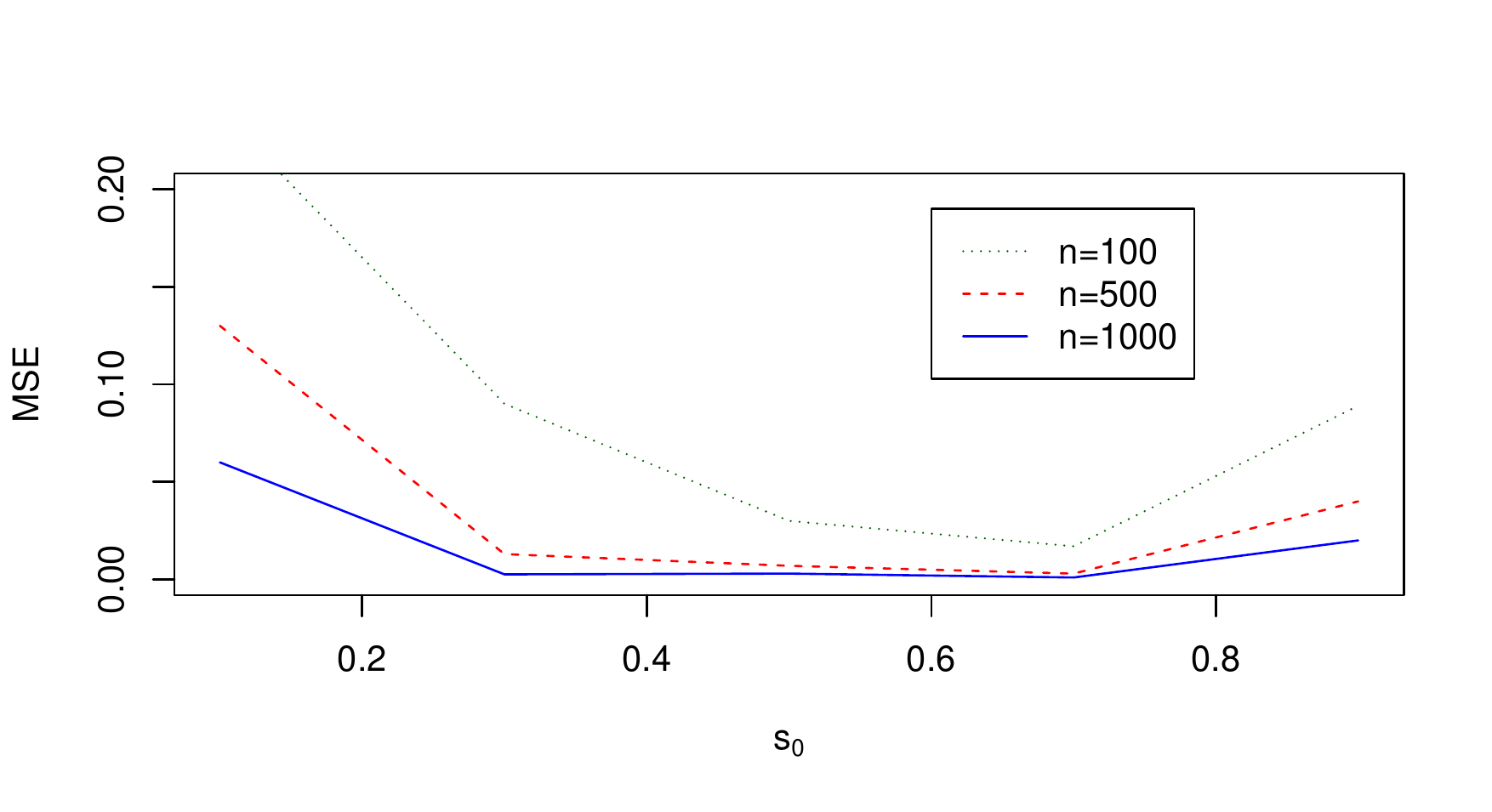}} 
    \subfigure{\includegraphics[width=0.49\textwidth]{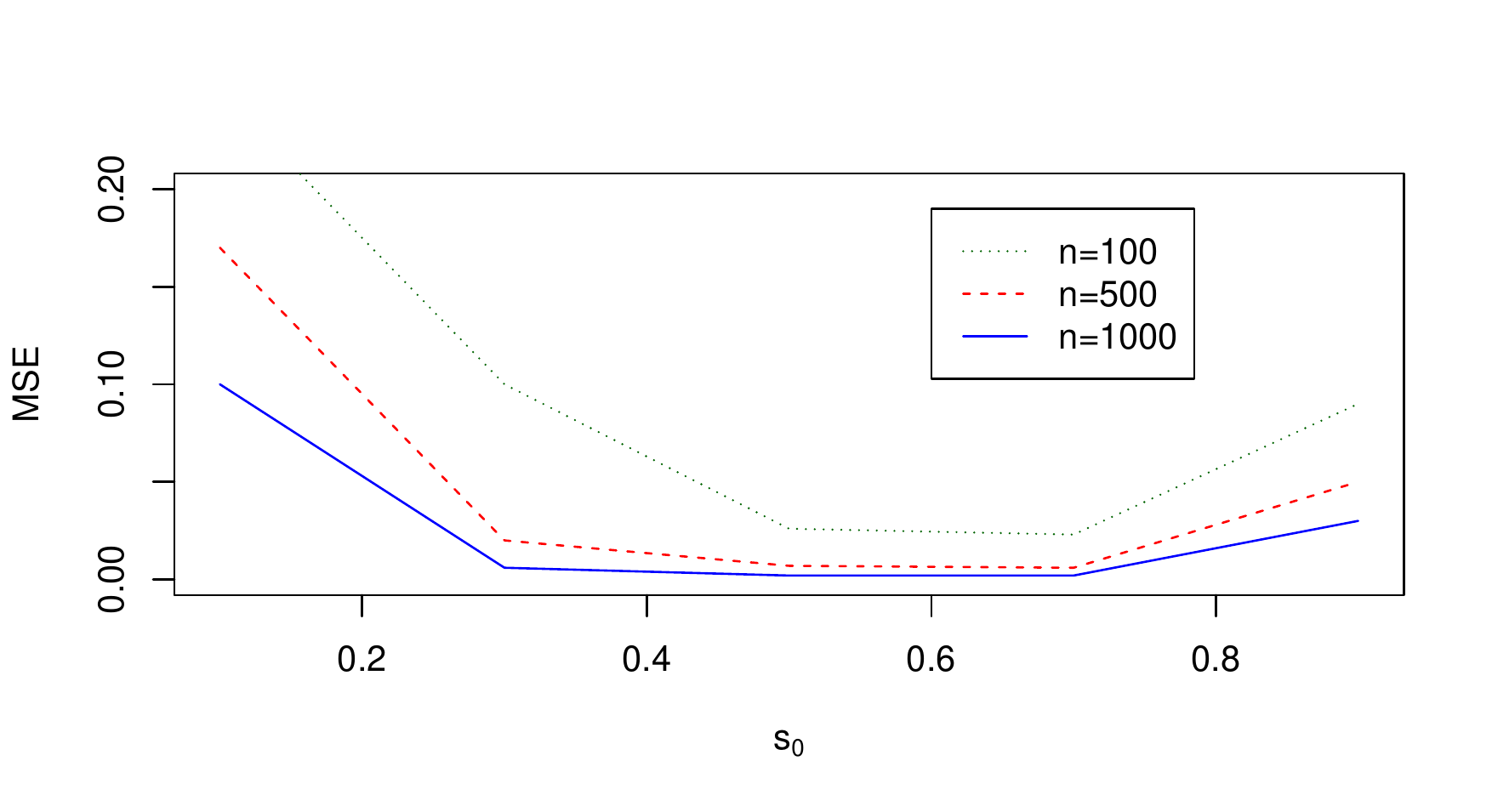}} 
  \caption{Simulation results for model (AR) and change point scenario (C1) (left), change point scenario (C2) (right)}
  \label{simus3}
\end{figure}

\begin{figure}[ht]
  \centering
  \subfigure{\includegraphics[width=0.49\textwidth]{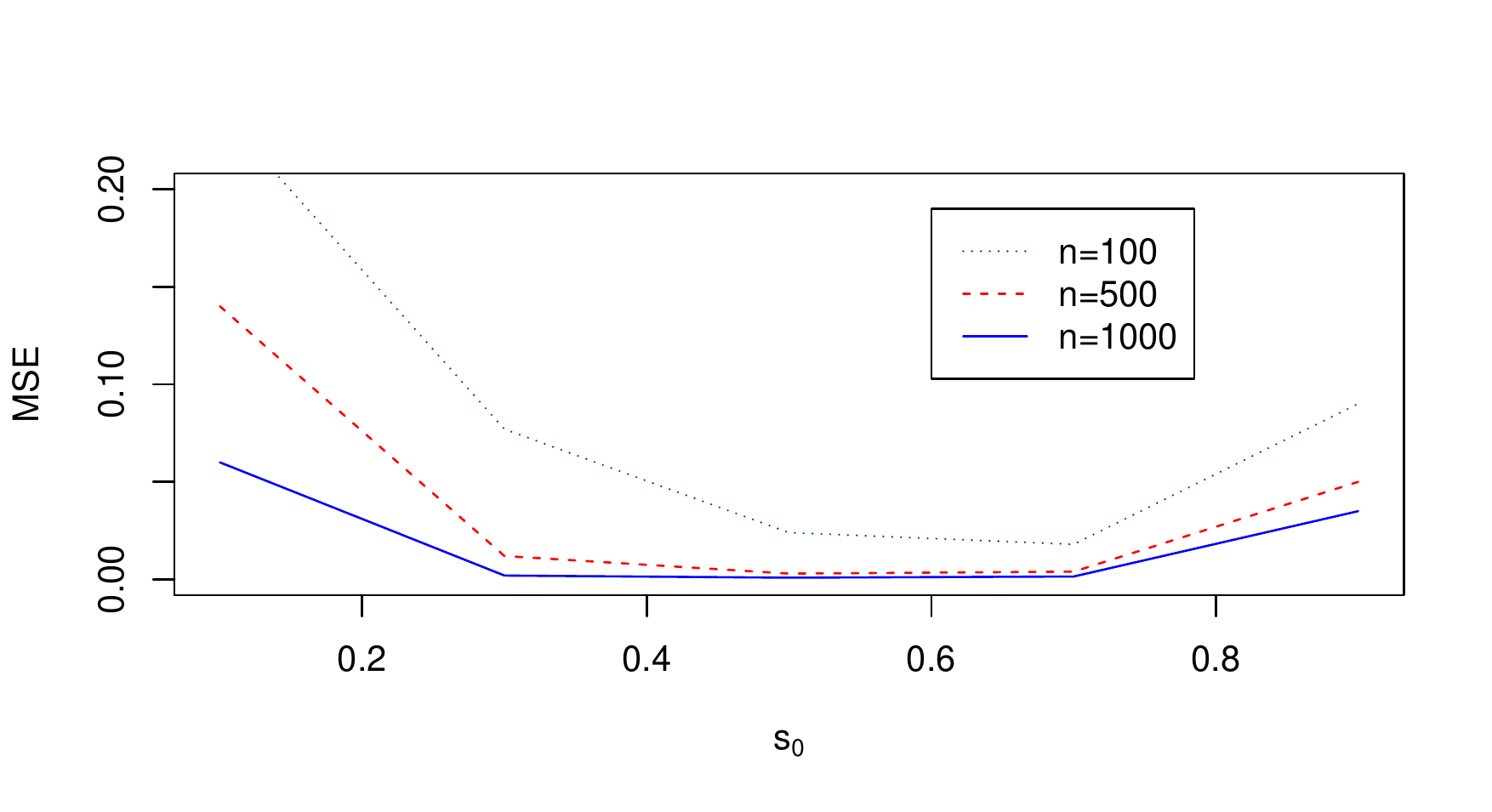}} 
    \subfigure{\includegraphics[width=0.49\textwidth]{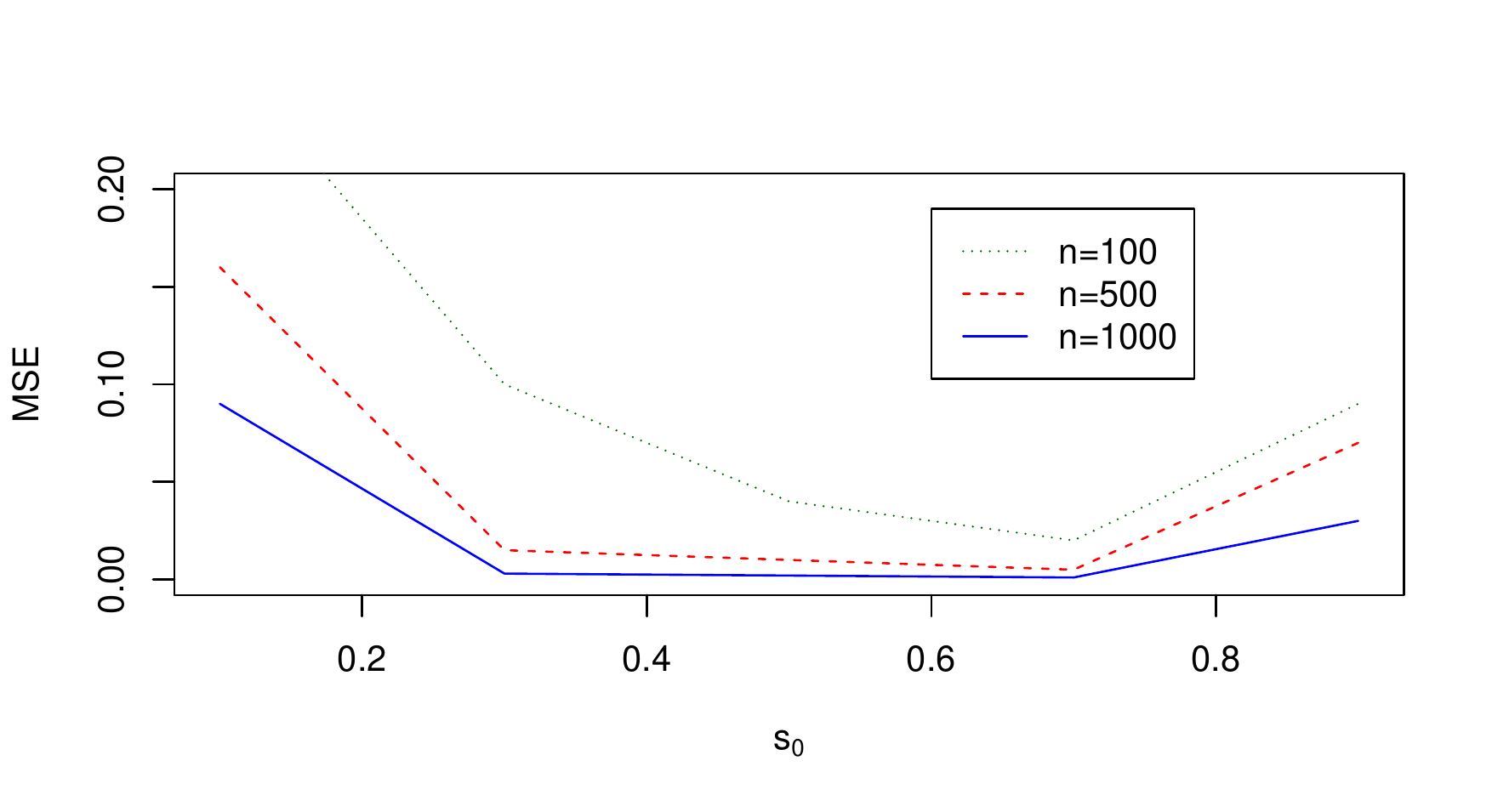}} 
  \caption{Simulation results for homoscedastic model (AR) with change point scenario (C3) (left) and heteroscedastic model (AR) with change point scenario (C1) (right)}
  \label{simus4}
\end{figure}

As stated in the remark in section \ref{sec:model} it is also possible to base the estimator on a Cramér-von Mises type functional of the marked empirical process of residuals. The simulation results for this type of estimator are very similar to those presented here for the Kolmogorov-Smirnov type estimator $\hat s_n$ and are omitted for the sake of brevity.

\subsection{Data example}

Finally, we will consider a real data example. The data at hand contains $36$ measurements of the annual flow volume of the small Czech river, R\'{a}ztoka, recorded between 1954 and 1989 as well as the annual rainfall during that time. It was considered by \cite{Huskova2003201} to investigate the effect of controlled deforestation on the capability for water retention of the soil. To this end it is of interest if and when the relationship between rainfall and flow volume changes. We set $X_t$ as the annual rainfall and $Y_t$ as the annual flow volume. \cite{Mohr2019} applied their Kolmogorov-Smirnov test to this data set, which clearly rejects the null of no change in the conditional mean function, indicating the existence of a change in the relationship between rainfall and flow volume. Using $\hat{s}_n$ to estimate the unknown time of change suggests a change in 1979. Note that this is consistent with the literature. As was pointed out by \cite{Huskova2003201} large scale deforestation had started around that time. Figure \ref{fig:creek_plot} shows on the left-hand side the scatterplot $X_t$ against $Y_t$ using dots for the observations after the estimated change and crosses for the observations before the estimated change. On the right-hand side the figure shows the cumulative sum, ${n}^{1/2}\sup_{z\in\R}|\hat{T}_n(\cdot,z)|$, as well as the critical value of the test used in \cite{Mohr2019} (red horizontal line) and the estimated change (green vertical line). Note that $\tilde{s}_n$ leads to the same result.

\begin{figure}[ht]
  \centering
     \includegraphics[width=0.8\textwidth]{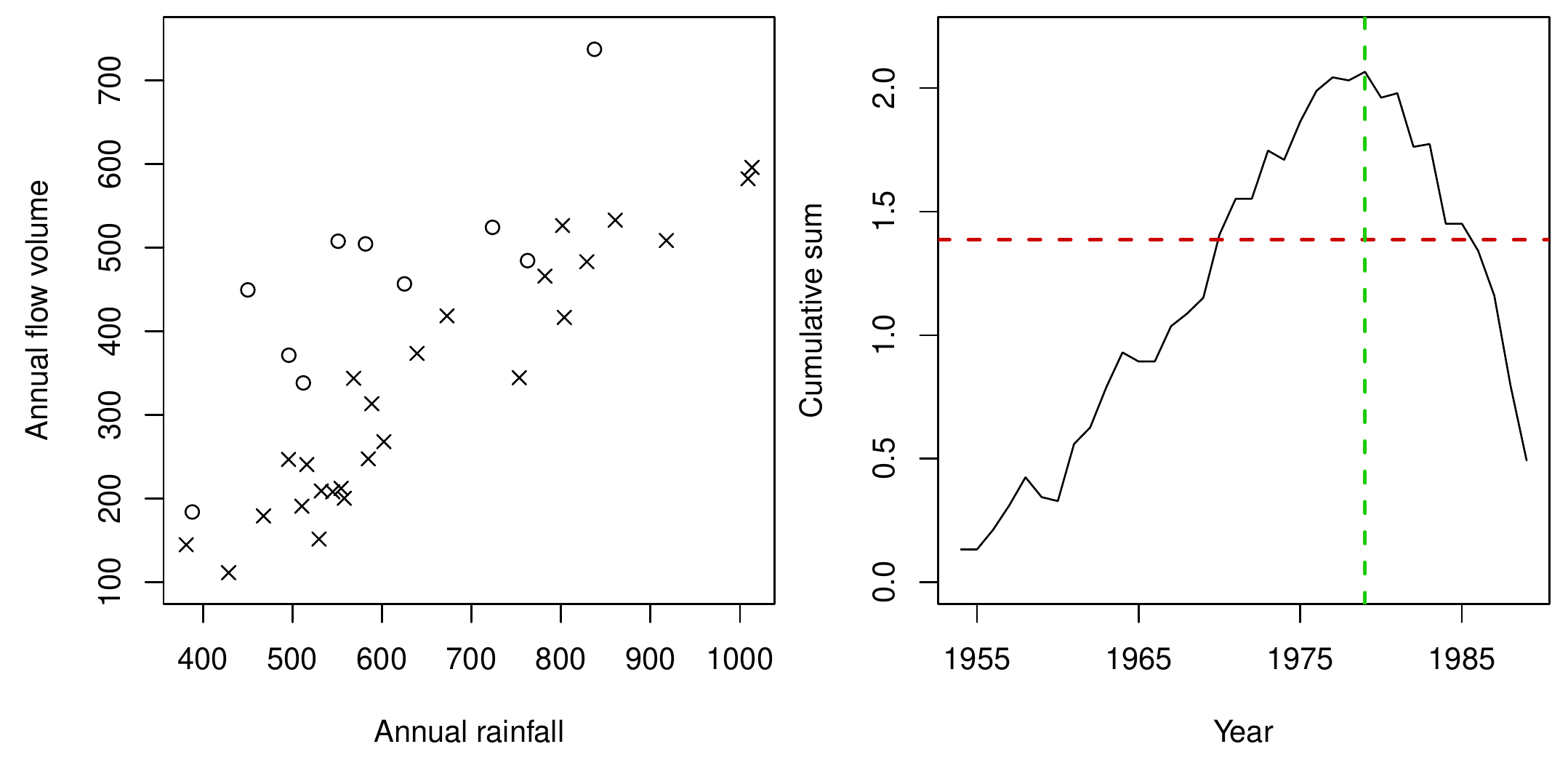}
  \caption{R\'{a}ztoka data: scatterplot (left) and CUSUM (right)}
  \label{fig:creek_plot}
\end{figure}

\section{Concluding remarks} \label{conclusion}

In this paper we consider nonparametric regression models with a change in the unknown regression function that allows for time series data as well as conditional heteroscedasticity. We propose an estimator for the rescaled change point that is based on the sequential marked empirical process of residuals and show consistency as well as a rate of convergence of $O_P(n^{-1})$. In an autoregressive setting we additionally give a consistency result for the proposed estimator. 

If more than one change occurs, the proposed estimator is not consistent for one of the changes in some situations. For detecting multiple changes we refer the reader to alternative procedures such as the MOSUM procedure proposed by \cite{Kirch2018526} or the wild binary segmentation procedure by \cite{Fryzlewicz20142243} (see also \cite{Fryzlewicz2019}).

Investigating the asymptotic distribution of the proposed estimator is a subsequent issue. Certainly, it is of great interest as it can be used to obtain confidence intervals. However, this subject goes beyond the scope of the paper at hand and is postponed to future research.

%
%
%

\appendix

\section{Proofs} \label{proofs}

\subsection{Auxiliary results}\label{auxiliaryresults}

\begin{lemma}\label{mdachRate}
Under the assumptions (P), (N), (J), (F), (K) and (B), it holds that
\[\sup_{\bm{x}\in\bm{J}_n}|\hat{m}_n(\bm{x})-\bar{m}_n(\bm{x})|=O_P\left(\left(\sqrt{\frac{\log(n)}{nh_n^d}}+h_np_n\right)\delta_np_nq_n\right),\]
where
\[\bar{m}_n(\bx)
=\frac{\sum_{i=1}^{n}f_{i}(\bm{x})m_{i}(\bm{x})}{\sum_{i=1}^{n}f_{i}(\bm{x})}.\]
\end{lemma}

The proof is similar to the proof of Lemma 2.2 in \cite{Mohr2018}. The key tool is an application of Theorem 1 in \cite{Kristensen20091433}. Details are omitted for the sake of brevity. 

\begin{Rem}
Under $(X1)$ we have
\[\bar{m}_n(\bx)
=\frac{\lf ns_0\rf}{n}m_{(1)}(\bx)+\frac{n-\lf ns_0 \rf}{n}m_{(2)}(\bx),\]
under (X2) and (X3) we have
\[\bar{m}_n(\bx)
=\frac{\frac{\lf ns_0\rf}{n}f_{(1)}(\bx)}{\bar{f}_{n}(\bm{x})}(m_{(1)}(\bx)-m_{(2)}(\bx))+m_{(2)}(\bx),\]
where 
\[\bar{f}_{n}(\bm{x}):=\frac{1}{n}\sum_{i=1}^{n}f_i(\bm{x})=
\begin{cases}
\frac{\lf ns_0\rf}{n}f_{(1)}(\bx)+\frac{n-\lf ns_0\rf}{n}f_{(2)}(\bx), & \text{for (X2)}\\
\frac{\lf ns_0\rf}{n}f_{(1)}(\bx)+\frac{n-\lf ns_0\rf}{n}f_{(2)}(\bx)+R_{n}(\bx), & \text{for (X3)}
\end{cases} \]
with $R_{n}(\cdot)$ from assumption (X3).
%
\end{Rem}


The proofs of the following lemmata can be found in appendix \ref{sec:prooflemma}.

\begin{lemma}\label{Urate}
Under the assumptions of Theorem \ref{cons} as well as under those of Theorem \ref{consAR} there exists a constant $\bar C=\bar C(C)<\infty$ such that
\[P\left(\sup_{s\in[0,1]}\sup_{z\in\er^d}\left| \sum_{i=L+1}^{L+\lf\kappa_n s\rf} U_{i}\omega_n(\bX_{i})I\{\bX_{i}\leq \bz\}\right|>C\kappa_n\right)\leq \bar C\kappa_n^{\frac 1q-1}\]
for all $L=0,1,\ldots,n-\kappa_n$, $1\leq\kappa_n\leq n$, $n\in\N$ and all $C>0$ with $q$ from assumption (U).
\end{lemma}


\begin{lemma}\label{mneg}
Under the assumptions of Theorem \ref{cons} as well as under those of Theorem \ref{consAR} there exists a constant $\bar C=\bar C(C)<\infty$ such that
\beq
&&P\Bigg(\sup_{s\in[0,1]}\sup_{z\in\er^d}\bigg| \sum_{i=L+1}^{(L+\lfloor \kappa_ns\rfloor)\wedge\lfloor ns_0\rfloor}\Big((m_{(1)}(\bX_{i})-\bar m_n(\bX_{i}))\omega_n(\bX_{i})I\{\bX_{i}\leq\bz\}\\
&&\qquad\qquad\qquad\quad-E\left[(m_{(1)}(\bX_{i})-\bar m_n(\bX_{i}))\omega_n(\bX_{i})I\{\bX_{i}\leq\bz\}\right]\Big)\bigg|>C\kappa_n\Bigg)\leq \bar C\kappa_n^{\frac 1r-1}
\eeq
and
\beq
&&P\Bigg(\sup_{s\in[0,1]}\sup_{z\in\er^d}\bigg|\sum_{i=L\vee\lfloor ns_0\rfloor+1}^{L+\lfloor \kappa_ns\rfloor}\Big((m_{(2)}(\bX_{i})-\bar m_n(\bX_{i}))\omega_n(\bX_{i})I\{\bX_{i}\leq\bz\}\\
&&\qquad\qquad\qquad\qquad\qquad-E\left[(m_{(2)}(\bX_{i})-\bar m_n(\bX_{i}))\omega_n(\bX_{i})I\{\bX_{i}\leq\bz\}\right]\Big)\bigg|>C\kappa_n\Bigg)\leq \bar C\kappa_n^{\frac 1r-1}
\eeq
for all $L=0,1,\ldots,n-\kappa_n$, $1\leq\kappa_n\leq n$, $n\in\N$ and all $C>0$ with $r$ from assumption (M).
\end{lemma}


\begin{lemma}\label{munif}
Under the assumptions of Theorem \ref{cons} as well as under those of Theorem \ref{consAR} it holds
\[P\left(\sup_{s\in[0,1]}\sup_{z\in\er^d}\left| \sum_{i=L+1}^{L+\lfloor \kappa_ns\rfloor} (\bar m_n(\bX_{i})-\hat m_n(\bX_{i}))\omega_n(\bX_{i})I\{\bX_{i}\leq\bz\}\right|>C\kappa_n\right)\leq C^{-1}\kappa_n^{-\zeta}\]
for all $L=0,1,\ldots,n-\kappa_n$, $1\leq\kappa_n\leq n$, $n\in\N$ and all $C>0$ with $\zeta>0$ from assumption (B).
\end{lemma}

\subsection{Proof of main results}\label{sec:prooftheorem}

We will proof Theorem \ref{cons} under the assumption (X1) and simply make a note on the parts that change under (X2).


\begin{proof}[Proof of Theorem \ref{cons}]
First note that for all $s\in[0,1]$ and $\bm{z}\in\R^d$
\begin{align}
\hat{T}_n(s,\bm{z})=A_n(s,\bm{z})+\Delta_{n,1}(s)\Delta_{n,2}(\bm{z}),\label{eq:decom}
\end{align}
where $A_{n}(s,\bm{z})=A_{n,1}(s,\bm{z})+A_{n,2}(s,\bm{z})+A_{n,3}(s,\bm{z})+A_{n,4}(s,\bm{z})$ with
\begin{align}
A_{n,1}(s,\bm{z})&:=\frac 1n\sins U_{i}\omega_n(\bX_{i})I\{\bX_{i}\leq\bz\}\label{eq:U}\\
A_{n,2}(s,\bm{z})&:=\frac 1n\sum_{i=1}^{\lfloor n(s\wedge s_0)\rfloor}\Big((m_{(1)}(\bX_{i})-\bar m_n(\bX_{i}))\omega_n(\bX_{i})I\{\bX_{i}\leq\bz\}\notag\\
&\qquad\qquad\quad-E\left[(m_{(1)}(\bX_{i})-\bar m_n(\bX_{i}))\omega_n(\bX_{i})I\{\bX_{i}\leq\bz\}\right]\Big)\label{eq:m1}\\
A_{n,3}(s,\bm{z})&:=I\{s>s_0\}\frac 1n\sum_{i=\lfloor ns_0\rfloor+1}^{\ns}\Big((m_{(2)}(\bX_{i})-\bar m_n(\bX_{i}))\omega_n(\bX_{i})I\{\bX_{i}\leq\bz\}\notag\\
&\qquad\qquad\qquad\qquad-E\left[(m_{(2)}(\bX_{i})-\bar m_n(\bX_{i}))\omega_n(\bX_{i})I\{\bX_{i}\leq\bz\}\right]\Big)\label{eq:m2}\\
A_{n,4}(s,\bm{z})&:=\frac 1n\sins (\bar m_n(\bX_{i})-\hat m_n(\bX_{i}))\omega_n(\bX_{i})I\{\bX_{i}\leq\bz\}\label{eq:munif}
\end{align}
and
\begin{align*}
\Delta_{n,1}(s)&:=I\{s\le s_0\}\frac{n-\lf ns_0 \rf}{n}\frac{\lf ns \rf}{n}+I\{s> s_0\}\frac{n-\lf ns \rf}{n}\frac{\lf ns_0 \rf}{n}\\
\Delta_{n,2}(\bm{z})&:=\int_{(-\bm{\infty},\bm{z}]}(m_{(1)}(\bm{x})-m_{(2)}(\bm{x}))f(\bm{x})\omega_n(\bx)d\bm{x},
\end{align*}
since by inserting the definition of $\bar{m}_n$ we obtain for $s\le s_0$
\begin{align*}
&\frac{1}{n}\sum\limits_{i=1}^{\lf ns\rf}E\left[(m_{(1)}(\bX_{i})-\bar m_n(\bX_{i}))\omega_n(\bX_{i})I\{\bX_{i}\leq\bz\}\right]\\
&\qquad=\frac{n-\lf ns_0\rf}{n}\frac{1}{n}\sum\limits_{i=1}^{\lf ns\rf}E\left[(m_{(1)}(\bX_{i})-m_{(2)}(\bX_{i}))\omega_n(\bX_{i})I\{\bX_{i}\leq\bz\}\right]\\
&\qquad=\frac{n-\lf ns_0\rf}{n}\frac{\lf ns\rf}{n}\Delta_{n,2}(\bm{z})
\end{align*}
and for $s>s_0$ 
\begin{align*}
&\frac{1}{n}\sum\limits_{i=1}^{\lf ns_0\rf}E\left[(m_{(1)}(\bX_{i})-\bar m_n(\bX_{i}))\omega_n(\bX_{i})I\{\bX_{i}\leq\bz\}\right]\\
&\qquad\qquad+\frac{1}{n}\sum\limits_{i=\lf ns_0\rf +1}^{\lf ns\rf}E\left[(m_{(2)}(\bX_{i})-\bar m_n(\bX_{i}))\omega_n(\bX_{i})I\{\bX_{i}\leq\bz\}\right]\\
&\qquad=\frac{n-\lf ns_0\rf}{n}\frac{1}{n}\sum\limits_{i=1}^{\lf ns_0\rf}E\left[(m_{(1)}(\bX_{i})-m_{(2)}(\bX_{i}))\omega_n(\bX_{i})I\{\bX_{i}\leq\bz\}\right]\\
&\qquad\qquad-\frac{\lf ns_0\rf}{n}\frac{1}{n}\sum\limits_{i=\lf ns_0\rf +1}^{\lf ns\rf}E\left[(m_{(1)}(\bX_{i})-m_{(2)}(\bX_{i}))\omega_n(\bX_{i})I\{\bX_{i}\leq\bz\}\right]\\
&\qquad=\frac{n-\lf ns\rf}{n}\frac{\lf ns_0\rf}{n}\Delta_{n,2}(\bm{z}).
\end{align*}
Note that we use the notation $\int_{(\bm{\infty},\bm{z}]}g(\bm{x})d\bm{x}=\int_{-\infty}^{z_d}\dots\int_{-\infty}^{z_1}g(x_1,\dots,x_d)dx_1\dots dx_d$ here. Due to the dominated convergence theorem and assumption (M), it holds that 
\[\Delta_{n,1}(s)\Delta_{n,2}(\bm{z})=\Delta_{1}(s)\Delta_{2}(\bm{z})+o(1),\]
uniformly in $s\in[0,1]$ and $\bm{z}\in\R^d$, where
\begin{align*}
\Delta_{1}(s)&:=I\{s\le s_0\}(1-s_0)s+I\{s>s_0\}(1-s)s_0,\\
\Delta_{2}(\bm{z})&:=\int_{(-\bm{\infty},\bm{z}]}(m_{(1)}(\bm{x})-m_{(2)}(\bm{x}))f(\bm{x})d\bm{x}.
\end{align*}
Note that under (X2) the same assertion holds with 
\begin{align*}
\Delta_{n,2}(\bm{z})&:=\int_{(-\bm{\infty},\bm{z}]}(m_{(1)}(\bm{x})-m_{(2)}(\bm{x}))\frac{f_{(1)}(\bm{x})f_{(2)}(\bm{x})}{\frac{\lf ns_0 \rf}{n}f_{(1)}(\bm{x})+\frac{n-\lf ns_0 \rf}{n}f_{(2)}(\bm{x})}\omega_n(\bx)d\bm{x}
\end{align*}
and
\begin{align*}
\Delta_{2}(\bz):=\int_{(-\bm{\infty},\bm{z}]}(m_{(1)}(\bm{x})-m_{(2)}(\bm{x}))\frac{f_{(1)}(\bm{x})f_{(2)}(\bm{x})}{s_0f_{(1)}(\bm{x})+(1-s_0)f_{(2)}(\bm{x})}d\bm{x}.
\end{align*}
By Lemma \ref{Urate}, Lemma \ref{mneg} and Lemma \ref{munif} with $\kappa_n=n$, it holds that $A_n(s,\bm{z})=o_P(1)$ uniformly in $s\in[0,1]$ and $\bm{z}\in\R^d$. Hence, we have shown that
\[\sup_{\bm{z}\in\R^d}|\hat{T}_n(s,\bm{z})|=\Delta_1(s)\sup_{\bm{z}\in\R^d}|\Delta_2(\bm{z})|+o_P(1)\]
%
%
uniformly in $s\in[0,1]$ under both cases (X1) and (X2). The assertion then follows by Theorem 2.12 in \cite{Kosorok2008} as $s_0$ is well-separated maximum of $[0,1]\to\R, \ s\mapsto \Delta_{1}(s)$.
\end{proof}

\begin{Rem}
Note that there are examples of $m_{(1)}$, $m_{(2)}$ and $f$ resp.\ $f_{(1)}$, $f_{(2)}$ that lead to $\Delta_{2}(\bm{\infty})=0$. In those cases a change point estimator based on the classical CUSUM $\hat T_n(s,\bm{\infty})$ is not consistent.
\end{Rem}


\begin{proof}[Proof of Theorem \ref{rates}]
First note that $s_0=\frac{\lf ns_0\rf}n+O(n^{-1})$ and $\hat s_n=\frac{\lf n\hat s_n\rf}n$. Thus we can consider $\left|\frac{\lf n\hat{s}_n\rf}n-\frac{\lf ns_0\rf}n\right|$ instead of $|\hat s_n-s_0|$.
The proof follows mainly along the same lines as the proof of Theorem 1 in \cite{Hariz20071802}. Consider the norm  $N:l^{\infty}(\R^d)\to \R, \ g\mapsto \sup_{z\in\R^d}|g(z)|$ and let $M>0$. 
We will show below that for all $\eta>0$ and $b,c>0$ 
it holds
\begin{align}
P\left(r_n\left|\frac{\lf n\hat{s}_n\rf}n-\frac{\lf ns_0\rf}n\right|>2^M\right)
&=P\left(r_n^{-1}2^M<\left|\frac{\lf n\hat{s}_n\rf}n-\frac{\lf ns_0\rf}n\right|\le \eta\right)\notag\\
&+P\left(\left|\frac{\lf n\hat{s}_n\rf}n-\frac{\lf ns_0\rf}n\right|>\eta\right)\notag\\
&\le E_{n,1} +E_{n,2} + E_{n,3}+E_{n,4},\label{eq:rates_01}
\end{align}
where 
%
\begin{align*}
E_{n,1}&:= P\left(r_n^{-1}2^M<\left|\frac{\lf n\hat{s}_n\rf}n-\frac{\lf ns_0\rf}n\right|\le \eta, N(A_{n}(\hat s_n,\cdot)-A_n(s_0,\cdot))\ge C\left|\frac{\lf n\hat{s}_n\rf}n-\frac{\lf ns_0\rf}n\right|\right)\\
E_{n,2}&:=P(N(A_n(s_0,\cdot))>c)\\
E_{n,3}&:=P(\Delta_{n,1}(s_0)N(\Delta_{n,2}(\cdot))\le b)\\
E_{n,4}&:=P\left(|\hat{s}_n-s_0|>\eta\right),
\end{align*}
with $C:=b-2c$. 
%
%
%
Now it holds that $E_{n,4}\to 0$ for all $\eta>0$, due to Theorem \ref{cons}. Further, $E_{n,2}\to 0$ for all $c>0$ as $A_n(s_0,\bm{z})=o_P(1)$ holds uniformly in $\bm{z}\in\R^d$. Finally choose $b>0$ and $n'=n'(b)\in\N$ such that $E_{n,3}= 0$ for all $n\ge n'$, which exists as $\Delta_{1}(s_0)N(\Delta_2(\cdot))>0$ and $\Delta_{n,1}(s_0)N(\Delta_{n,2}(\cdot))=\Delta_{1}(s_0)N(\Delta_2(\cdot))+o(1)$. We then choose $c>0$ such that $b-2c>0$.
To see the validity of \eqref{eq:rates_01} first note that for all $s\in[0,1]$
\begin{align*}
\hat{T}_n(s,\cdot)
&=A_{n}(s,\cdot)+\Delta_{n,1}(s)\Delta_{n,2}(\cdot)\\
&=A_{n}(s,\cdot)-A_{n}(s_0,\cdot)+A_{n}(s_0,\cdot)\left(1-\frac{\Delta_{n,1}(s)}{\Delta_{n,1}(s_0)}\right)+\frac{\Delta_{n,1}(s)}{\Delta_{n,1}(s_0)}\hat{T}_n(s_0,\cdot).
\end{align*}
Applying the norm and triangular inequality we obtain for all $s\in[0,1]$
\begin{align*}
N(\hat{T}_n(s,\cdot))
&\le N(A_{n}(s,\cdot)-A_{n}(s_0,\cdot))+\left(1-\frac{\Delta_{n,1}(s)}{\Delta_{n,1}(s_0)}\right)N(A_{n}(s_0,\cdot))+\left(\frac{\Delta_{n,1}(s)}{\Delta_{n,1}(s_0)}\right)N(\hat{T}_n(s_0,\cdot))
\end{align*}
which is equivalent to
\begin{align*}
N(\hat{T}_n(s,\cdot))-N(\hat{T}_n(s_0,\cdot))
&\le N(A_{n}(s,\cdot)-A_{n}(s_0,\cdot))+\left(\frac{\Delta_{n,1}(s)}{\Delta_{n,1}(s_0)}-1\right)\left(N(\hat{T}_n(s_0,\cdot))-N(A_{n}(s_0,\cdot))\right).
\end{align*}
Due to the definition of $\hat{s}_n$ it holds that $N(\hat{T}_n(\hat{s}_n,\cdot))-N(\hat{T}_n(s_0,\cdot))\ge 0$. Additionally using 
the specific definition of $\Delta_{n,1}$ we obtain
\begin{align*}
N(A_{n}(\hat{s}_n,\cdot)-A_{n}(s_0,\cdot))
&\ge \left(1-\frac{\Delta_{n,1}(\hat{s}_n)}{\Delta_{n,1}(s_0)}\right)\left(N(\hat{T}_n(s_0,\cdot))-N(A_{n}(s_0,\cdot))\right)\\
&\ge \underbrace{\min\left(\frac n{\lf ns_0\rf},\frac n{n-\lf ns_0\rf}\right)}_{> 1}\left|\frac{\lf n\hat{s}_n\rf}n-\frac{\lf ns_0\rf}n\right|\left(N(\hat{T}_n(s_0,\cdot))-N(A_{n}(s_0,\cdot))\right)\\
&\ge \left|\frac{\lf n\hat{s}_n\rf}n-\frac{\lf ns_0\rf}n\right|\left(\Delta_{n,1}(s_0)N(\Delta_{n,2}(\cdot))-2N(A_n(s_0,\cdot))\right),
\end{align*}
where we again make use of the triangular inequality in the last step. Putting the results together we obtain
\begin{align*}
&P\left(r_n^{-1}2^M<\left|\frac{\lf n\hat{s}_n\rf}n-\frac{\lf ns_0\rf}n\right|\le \eta\right)\\
&\le P\left(r_n^{-1}2^M<\left|\frac{\lf n\hat{s}_n\rf}n-\frac{\lf ns_0\rf}n\right|\le \eta, \Delta_{n,1}(s_0)N(\Delta_{n,2}(\cdot))>b,N(A_n(s_0,\cdot))\le c\right)\\
&\hspace{1cm}+P(\Delta_{n,1}(s_0)N(\Delta_{n,2}(\cdot))\le b)+P(N(A_n(s_0,\cdot))>c)\\
&\le P\left(r_n^{-1}2^M<\left|\frac{\lf n\hat{s}_n\rf}n-\frac{\lf ns_0\rf}n\right|\le \eta, N(A_{n}(\hat s_n,\cdot)-A_n(s_0,\cdot))\ge C\left|\frac{\lf n\hat{s}_n\rf}n-\frac{\lf ns_0\rf}n\right|\right)\\
&\hspace{1cm}+P(\Delta_{n,1}(s_0)N(\Delta_{n,2}(\cdot))\le b)+P(N(A_n(s_0,\cdot))>c).
\end{align*}
Finally we will investigate $E_{n,1}$. To do this we define shells
\[S_{n,l}=\left\{t\in[0,1]:2^l<r_n\left|t-\frac{\lf ns_0\rf}n\right| \leq 2^{l+1}\right\}\]
and choose $L_n=L_n(\eta)$ such that $2^{L_n}<r_n\eta\leq 2^{L_n+1}$ for some $\eta\leq \frac 12$.
Then
\beq
E_{n,1}&\leq& \sum_{l=M}^{L_n} P\left(\frac{\lf n\hat s_n\rf}n\in S_{n,l}, N(A_{n}(\hat s_n,\cdot)-A_n(s_0,\cdot))\ge C\left|\frac{\lf n\hat{s}_n\rf}n-\frac{\lf ns_0\rf}n\right|\right)\\
&\leq&\sum_{l=M}^{L_n} P\left(\sup_{s:\left|\frac{\lf ns\rf}n-\frac{\lf ns_0\rf}n\right|\leq 2^{l+1}r_n^{-1}}N(A_{n}(s,\cdot)-A_n(s_0,\cdot))\ge C2^lr_n^{-1}\right)\\
&\leq&\sum_{l=M}^{L_n}\sum_{i=1}^4 P\left(\sup_{s:\left|\frac{\lf ns\rf}n-\frac{\lf ns_0\rf}n\right|\leq 2^{l+1}r_n^{-1}}N(A_{n,i}(s,\cdot)-A_{n,i}(s_0,\cdot))\ge \frac C42^lr_n^{-1}\right)\\
&\leq&\tilde C \left(\left(\frac {n}{r_n}\right)^{\frac 1q-1}\sum_{l=M}^{L_n}(2^{\frac 1q-1})^l+\left(\frac {n}{r_n}\right)^{{\frac 1r-1}}\sum_{l=M}^{L_n}(2^{\frac 1r-1})^l+\left(\frac {n}{r_n}\right)^{-\zeta}\sum_{l=M}^{L_n}(2^{-\zeta})^l\right)
\eeq
for some constant $\tilde C<\infty$ by Lemmata \ref{Urate}, \ref{mneg} and \ref{munif} with $\kappa_n=\lf 2^{l+1}\frac{n}{r_n}\rf$ with $q$ from assumption (U), $r$ from assumption (M) and $\zeta>0$ from assumption (B). Now choosing $r_n=n$ and letting $n$ and thus $L_n$ tend to infinity and then $M$ to infinity, the assertion of Theorem \ref{rates} follows.
%
%
%
%
\end{proof}


\begin{proof}[Proof of Theorem \ref{consAR}]
Under (X3) we have for all $s\in[0,1]$ and $z\in\R$
\begin{align*}
\hat{T}_n(s,z)=A_n(s,z)+\Delta_{n,1}(s)\Delta_{n,2}(z)+\tilde{\Delta}_n(s,z),
\end{align*}
with $A_n(s,z)$ and $\Delta_{n,1}(s)$ from the proof of Theorem \ref{cons}, and with
\[\Delta_{n,2}(z)
:=\int_{(-\infty,z]}(m_{(1)}(x)-m_{(2)}(x))\frac{f_{(1)}(x)f_{(2)}(x)}{\frac{\lf ns_0\rf}{n}f_{(1)}(x)+\frac{n-\lf ns_0\rf}{n}f_{(2)}(x)+R_{n}(x)}\omega_n(x)dx
\]
and
\[\tilde{\Delta}_n(s,z)
:=\int_{(-\infty,z]}(m_{(1)}(x)-m_{(2)}(x))I\{s\le s_0\}\frac{\lf ns\rf}{n}\frac{f_{(1)}(x)R_{n}(x)}{\frac{\lf ns_0\rf}{n}f_{(1)}(x)+\frac{n-\lf ns_0\rf}{n}f_{(2)}(x)+R_{n}(x)}\omega_n(x)dx.
\]
Now it holds that
\[\Delta_{n,2}(z)\to\int_{(-\infty,z]}(m_{(1)}(x)-m_{(2)}(x))\frac{f_{(1)}(x)f_{(2)}(x)}{s_0 f_{(1)}(x)+(1-s_0)f_{(2)}(x)}dx=:\Delta_2(z) \]
and $\tilde{\Delta}_n(s,z)\to 0$ uniformly in $s\in[0,1]$ and $z\in\R$, due to dominated convergence and assumption (M). Hence we have uniformly in $s$ and $z$
\[\hat{T}_n(s,z)=A_n(s,z)+\Delta_{1}(s)\Delta_{2}(z)+o(1),\]
with $\Delta_{1}(s)$ as in the proof of Theorem \ref{cons}. The rest goes analogously to the proof of Theorem \ref{cons}.
\end{proof}

\begin{Rem}
Note that for finite $n\in\N$ we do not get the decomposition of $\hat{T}_n$ as in \eqref{eq:decom} in the proof of Theorem \ref{cons}. We only obtain this kind of decomposition when letting $n$ tend to infinity. 
The decomposition for finite $n$, however, is essential for the proof of the rates of convergence in Theorem \ref{rates}. 
\end{Rem}

\subsection{Proofs of lemmata}\label{sec:prooflemma}


\begin{proof}[Proof of Lemma \ref{Urate}]
The proof follows along similar lines as the proof of Lemma A.3 in \cite{Mohr2018}. Throughout the proof the values of $C$ and $\bar C$ may vary from line to line but they are always positive, finite and independent of $n$. Further note that deterministic terms that are of order $O(\kappa_n)$ can be omitted as we can choose constants appropriately. It holds that
\begin{eqnarray}
\nn&&\sup_{s\in[0,1]}\sup_{\bz\in\er^d}\left|\sinsk U_{i}\omega_n(\bX_{i})I\{\bX_{i}\leq \bz\}\right|\\
\nn&=&\sup_{s\in[0,1]}\sup_{\bz\in\er^d}\left| \sinsk U_{i}\omega_n(\bX_{i})I\{\bX_{i}\leq \bz\}-E\left[\sinsk U_{i}\omega_n(\bX_{i})I\{\bX_{i}\leq \bz\}\right]\right|\\
\nn&\leq&\sup_{s\in[0,1]}\sup_{\bz\in\er^d}\left| \sinsk U_{i}I\{|U_{i}|> \kappa_n^{\frac 1q}\}\omega_n(\bX_{i})I\{\bX_{i}\leq \bz\}\right.\\
\label{Urate1}&&\qquad\qquad\qquad\left.-E\left[\sinsk U_{i}I\{|U_{i}|> \kappa_n^{\frac 1q}\}\omega_n(\bX_{i})I\{\bX_{i}\leq \bz\}\right]\right|\\
\nn&&+\sup_{s\in[0,1]}\sup_{\bz\in\er^d}\left| \sinsk U_{i}I\{|U_{i}|\leq \kappa_n^{\frac 1q}\}\omega_n(\bX_{i})I\{\bX_{i}\leq \bz\}\right.\\
\label{Urate2}&&\qquad\qquad\qquad\left.-E\left[\sinsk U_{i}I\{|U_{i}|\leq \kappa_n^{\frac 1q}\}\omega_n(\bX_{i})I\{\bX_{i}\leq \bz\}\right]\right|
\end{eqnarray}
where (\ref{Urate1}) is of the desired rate in probability since
\beq
P\left(\sum_{i=L+1}^{L+\kappa_n}|U_{i}|I\{|U_{i}|>\kappa_n^{\frac 1q}\}>C\kappa_n\right)&\leq&C^{-1}C_U\kappa_n^{\frac {1}q-1}
\eeq
by Markov's inequality with
\beq
E\left[|U_{i}|I\{|U_{i}|>\kappa_n^{\frac 1q}\}\right]&=&E\left[|U_{i}|^q|U_{i}|^{-(q-1)}I\{|U_{i}|>\kappa_n^{\frac 1q}\}\right]\\
&\leq&\kappa_n^{-\frac{q-1}q}E[|U_{i}|^q]\\
&\le&C_U\kappa_n^{\frac {1}q-1}\qquad \text{for all $i$ and for $C_U<\infty$ from assumption (U)}.
\eeq

\noindent Considering the term (\ref{Urate2}) we define the function class
\[\cf_n:=\left\{(u,\bx)\mapsto uI\{|u|\leq \kappa_n^{\frac 1q}\}\omega_n(\bx)I\{\bx\leq\bz\}:\bz \in \er^d\right\}\]
to rewrite the assertion as
\[P\left(\sup_{s\in[0,1]}\sup_{\varphi\in\cf_n}\left|\sinsk\left(\varphi(U_{i},\bX_{i})-\int \varphi dP\right)\right|>C\kappa_n\right)\leq \bar C\kappa_n^{\frac 1q-1}.\]
Now we will cover $[0,1]$ by finitely many intervals and $\cf_n$ by finitely many brackets to replace the supremum by a maximum. Let therefore 
\[0=s_1<\ldots<s_{K_n}=1\]
part the interval $[0,1]$ in $K_n$ subintervals of length $\bar\epsilon_n$ with $\bar\epsilon_n=\kappa_n^{- \frac 1q}$. Then
\beq
&&\sup_{s\in[0,1]}\sup_{\varphi\in\cf_n}\left|\sinsk\left(\varphi(U_{i},\bX_{i})-\int \varphi dP\right)\right|\\
&=&\max_k\sup_{\substack{s\in[0,1]\\|s-s_k|\leq\be}}\sup_{\varphi\in\cf_n}\left|\sinsk\left(\varphi(U_{i},\bX_{i})-\int \varphi dP\right)\right|\\
&\leq&\max_k\sup_{\varphi\in\cf_n}\left|\sum_{i=L+1}^{L+\lfloor \kappa_ns_k\rfloor}\left(\varphi(U_{i},\bX_{i})-\int \varphi dP\right)\right|\\
&&+\max_k\sup_{\substack{s\in[0,1]\\|s-s_k|\leq\be}}\sup_{\varphi\in\cf_n}\sum_{i=L+1}^{L+\kappa_n}\underbrace{\left|\varphi(U_{i},\bX_{i})-\int \varphi dP\right|}_{\leq 2\kappa_n^{\frac 1q}}\left|I\left\{\frac {i-L}{\kappa_n}\leq s\right\}-I\left\{\frac {i-L}{\kappa_n}\leq s_k\right\}\right|\\
&\leq&\max_k\sup_{\varphi\in\cf_n}\left|\sum_{i=L+1}^{L+\lfloor \kappa_ns_k\rfloor}\left(\varphi(U_{i},\bX_{i})-\int \varphi dP\right)\right|+2\kappa_n^{\frac 1q}\left(\kappa_n\bar{\epsilon}_n+1\right)
\eeq
and $2\kappa_n^{\frac 1q}\left(\kappa_n\bar{\epsilon}_n+1\right)=2(\kappa_n+\kappa_n^{\frac{1}{q}})=O(\kappa_n)$.
Further let
\[\varphi_{\bj}^u(u,\bx):=uI\{|u|\leq \kappa_n^{\frac 1q}\}I\{u\geq 0\}\omega_n(\bx)I\{\bx\leq\bz_{\bj}\}+uI\{|u|\leq \kappa_n^{\frac 1q}\}I\{u<0\}\omega_n(\bx)I\{\bx\leq\bz_{\bm{j-1}}\}\]
and
\[\varphi_{\bj}^l(u,\bx):=uI\{|u|\leq \kappa_n^{\frac 1q}\}I\{u\geq 0\}\omega_n(\bx)I\{\bx\leq\bz_{\bm {j-1}}\}+uI\{|u|\leq \kappa_n^{\frac 1q}\}I\{u<0\}\omega_n(\bx)I\{\bx\leq\bz_{\bm{j}}\}\]
form the brackets $[\varphi_{\bj}^l,\varphi_\bj^u]_{\bj\in\times_{i=1}^d\{1,\ldots,N_i\}}$ of $\cf_n$, where $\bz_\bj=(z_{j_1,1},\ldots,z_{j_d,d})$ and 
\[-\infty=z_{0,i}<\ldots<z_{N_i,i}=\infty\]
gives a partition of $\er$ for all $i=1,\ldots,d$. The total number of brackets $J_n:=N_{[\ ]}(\epsilon_n,\cf_n,\|\cdot\|_{L_1(P)})$ needed to cover $\cf_n$ is of order
$J_n=O(\epsilon_n^{-d})$, which follows analogously to but easier than the proof of Lemma A.7 in \cite{Mohr2018}.

\noindent For all $\varphi\in\cf_n$ there exists a $\bj$ with $\varphi_\bj^l\leq\varphi\leq\varphi_\bj^u$ and thus
\[\varphi-\int\varphi dP\leq\varphi_\bj^u-\int\varphi_\bj^udP+\int(\varphi_\bj^u-\varphi_\bj^l)dP\]
and
\[\varphi-\int\varphi dP\geq\varphi_\bj^l-\int\varphi_\bj^ldP-\int(\varphi_\bj^u-\varphi_\bj^l)dP.\]
Therefore for all $s\in[0,1]$
\beq
&&\sup_{\varphi\in\cf_n}\left|\sinsk\left(\varphi(U_{i},\bX_{i})-\int \varphi dP\right)\right|\\
&=&\max_{\bj}\sup_{\varphi\in[\varphi_\bj^l,\varphi_\bj^u]}\left|\sinsk\left(\varphi(U_{i},\bX_{i})-\int \varphi dP\right)\right|\\
&\leq&\max_{\bj}\max\left\{\left|\sinsk\left(\varphi_\bj^u(U_{i},\bX_{i})-\int \varphi_\bj^u dP\right)\right|,\left|\sinsk\left(\varphi_\bj^l(U_{i},\bX_{i})-\int \varphi_\bj^l dP\right)\right|\right\}\\
&&+\kappa_n\max_\bj\underbrace{\int(\varphi_\bj^u-\varphi_\bj^l)dP}_{ \leq\epsilon_n}
\eeq
and $\kappa_n\epsilon_{n}=O(\kappa_n)$ if we choose $\epsilon_n$ constant. Thus it remains to show that
\beq
&&P\left(\max_{\bj,k}\left|\sum_{i=L+1}^{L+\lfloor \kappa_ns_k\rfloor}\left(\varphi_\bj^u(U_{i},\bX_{i})-\int \varphi_\bj^u dP\right)\right|>C\kappa_n\right)\leq \bar C\kappa_n^{\frac 1q-1}
\eeq 
and the same with $\varphi_\bj^u$ replaced by $\varphi_\bj^l$. Recall that
\begin{eqnarray}
\nn&&\max_{\bj,k}\left|\sum_{i=L+1}^{L+\lfloor \kappa_ns_k\rfloor}\left(\varphi_\bj^u(U_{i},\bX_{i})-\int \varphi_\bj^u dP\right)\right|\\
\nn&\leq&\max_{\bj,k}\bigg|\sum_{i=L+1}^{L+\lfloor \kappa_ns_k\rfloor}\Big(U_{i}I\{|U_{i}|\leq \kappa_n^{\frac 1q}\}I\{U_{i}\geq 0\}\omega_n(\bX_{i})I\{\bX_{i}\leq\bz_\bj\}\\
\label{Uneg3}&&\qquad\qquad -E\left[U_{i}I\{|U_{i}|\leq \kappa_n^{\frac 1q}\}I\{U_{i}\geq 0\}\omega_n(\bX_{i})I\{\bX_{i}\leq\bz_\bj\}\right]\Big)\bigg|\\
\nn&&+\max_{\bj,k}\bigg|\sum_{i=L+1}^{L+\lfloor \kappa_ns_k\rfloor}\Big(U_{i}I\{|U_{i}|\leq \kappa_n^{\frac 1q}\}I\{U_{i}<0\}\omega_n(\bX_{i})I\{\bX_{i}\leq\bz_{\bm{j-1}}\}\\
\nn&&\qquad\qquad -E\left[U_{i}I\{|U_{i}|\leq \kappa_n^{\frac 1q}\}I\{U_{i}< 0\}\omega_n(\bX_{i})I\{\bX_{i}\leq\bz_{\bm {j-1}}\}\right]\Big)\bigg|.
\end{eqnarray}
We will only consider the first summand in more detail since the rest works analogously.

\noindent To prove that (\ref{Uneg3}) is stochastically of the desired rate we apply a Bernstein type inequality for $\alpha$-mixing processes, see \cite{Liebscher199669} Therorem 2.1. Following his notation we define
\beq
Z_i&:=&\Big(U_{i+L}I\{|U_{i+L}|\leq \kappa_n^{\frac 1q}\}I\{U_{i+L}\geq 0\}\omega_n(\bX_{i+L})I\{\bX_{i+L}\leq\bz\}\\
&&-E\left[U_{i+L}I\{|U_{i+L}|\leq \kappa_n^{\frac 1q}\}I\{U_{i+L}\geq 0\}\omega_n(\bX_{i+L})I\{\bX_{i+L}\leq\bz\}\right]\Big)I\left\{\frac i{\kappa_n}\leq s_k\right\}
\eeq
for fixed $\bz\in\er^d$ and $s\in[0,1]$. Note that $S(\kappa_n):=|Z_i|\leq 2\kappa_n^{\frac 1q}$, $Z_i$ is centered and $$D(\kappa_n,N):=\sup_{0\leq T\leq \kappa_n-1}E\left[\left(\sum_{j=T+1}^{(T+N)\wedge \kappa_n}Z_j\right)^2\right]\leq N^2E[Z_i^2]\leq C_UN^2$$ by assumption (U). Thus Liebscher's Theorem can be applied with $N=\lfloor \kappa_n^{1-\frac 2q}\rfloor$. This means that 
\beq
&&P\bigg(\max_{\bj,k}\bigg|\sum_{i=L+1}^{L+\lfloor \kappa_ns_k\rfloor}U_{i}I\{|U_{i}|\leq \kappa_n^{\frac 1q}\}I\{U_{i}\geq 0\}\omega_n(\bX_{i})I\{\bX_{i}\leq\bz_\bj\}\\
&&\qquad\qquad-E\left[U_{i}I\{|U_{i}|\leq \kappa_n^{\frac 1q}\}I\{U_{i}\geq 0\}\omega_n(\bX_{i})I\{\bX_{i}\leq\bz_{\bm {j}}\}\right]\bigg|>C\kappa_n\bigg)\\
&\leq&\sum_{\bj,k} P\bigg(\bigg|\sum_{i=L+1}^{L+\lfloor \kappa_ns_k\rfloor}\Big(U_{i}I\{|U_{i}|\leq \kappa_n^{\frac 1q}\}I\{U_{i}\geq 0\}\omega_n(\bX_{i})I\{\bX_{i}\leq\bz_\bj\}\\
&&\qquad\qquad-E\left[U_{i}I\{|U_{i}|\leq \kappa_n^{\frac 1q}\}I\{U_{i}\geq0\}\omega_n(\bX_{i})I\{\bX_{i}\leq\bz_{\bm {j}}\}\right]\Big)\bigg|>C\kappa_n\bigg)\\
&\leq& J_nK_n\left(4\exp\left(-\frac{C^2\kappa_n^2}{64\frac {\kappa_n}ND(\kappa_n,N)+\frac{8}3C\kappa_n NS(\kappa_n)}\right)+4\frac {\kappa_n}N\alpha(N)\right)\\
&\leq&J_nK_n\left(4\exp\left(-\frac{C^2\kappa_n^2}{64C_U\kappa_n^{2-\frac 2q}+\frac {16}3C\kappa_n^{2-\frac 1q}}\right)+4\kappa_n^{\frac 2q}\alpha(\kappa_n^{1-\frac 2q})\right)\\
&\leq&J_nK_n\left(4\exp\left(-C_1\kappa_n^{\frac 1q}\right)+4\kappa_n^{\frac 2q}\alpha(\kappa_n^{1-\frac 2q})\right)\\
&\leq&C_2\kappa_n^{\frac 1q}\left((C_1\kappa_n^{\frac 1q})^{-q}+\kappa_n^{\frac 2q-\bar\alpha+\frac{2\bar\alpha}q}\right)\\
&\leq&\bar C\kappa_n^{\frac 1q-1}
\eeq
for some constants $C_1,C_2,\bar C$ where the second to last inequality follows from the fact that $\exp(-x)<x^{-k}k!$ for all $k\in\en$ and $x\in\er_{>0}$ and the last inequality is true by assumption (P) which implies $\bar{\alpha}>(q+2)/(q-2)$. This completes the proof.
\end{proof}


\begin{proof}[Proof of Lemma \ref{mneg}]
First we will distinguish between the cases $L+\lfloor \kappa_n s\rfloor\leq \lfloor ns_0\rfloor$ and $L+\lfloor \kappa_n s\rfloor> \lfloor ns_0\rfloor$. In the first case we can write
\beq
&&\sum_{i=L+1}^{L+\lfloor \kappa_ns\rfloor}(m_{(1)}(\bX_{i})-\bar m_n(\bX_{i}))\omega_n(\bX_{i})I\{\bX_{i}\leq\bz\}\\\
&=&\sum_{i=L+1}^{L+\lfloor \kappa_ns\rfloor}\bigg(m_{(1)}(\bX_{i})-\frac{m_{(1)}(\bX_{i})\sum_{j=1}^{\lfloor ns_0\rfloor} f_{j}(\bX_{i})}{\sjn f_{j}(\bX_{i})}-\frac{m_{(2)}(\bX_{i})\sum_{j=\lfloor ns_0\rfloor+1}^n f_{j}(\bX_{i})}{\sjn f_{j}(\bX_{i})}\bigg)\\
&&\qquad\qquad\qquad\qquad\qquad\qquad\qquad\qquad\qquad\qquad\qquad\qquad\qquad\qquad\qquad\cdot \omega_n(\bX_{i})I\{\bX_{i}\leq\bz\}\\
&=&\sum_{i=L+1}^{L+\lfloor \kappa_ns\rfloor}(m_{(1)}(\bX_{i})-m_{(2)}(\bX_{i}))\frac{\sum_{j=\lfloor ns_0\rfloor+1}^n f_{j}(\bX_{i})}{\sjn f_{j}(\bX_{i})}\omega_n(\bX_{i})I\{\bX_{i}\leq\bz\}\
\eeq
and analogously for the second case
\beq
&&\sum_{i=L\vee \lfloor ns_0\rfloor+1}^{L+\lfloor\kappa_ns\rfloor}(m_{(2)}(\bX_{i})-\bar m_n(\bX_{i}))\omega_n(\bX_{i})I\{\bX_{i}\leq\bz\}\\\
&=&\sum_{i=L\vee\lfloor ns_0\rfloor+1}^{L+\lfloor \kappa_ns\rfloor}(m_{(2)}(\bX_{i})-m_{(1)}(\bX_{i}))\frac{\sum_{j=1}^{\lfloor ns_0\rfloor} f_{j}(\bX_{i})}{\sjn f_{j}(\bX_{i})}\omega_n(\bX_{i})I\{\bX_{i}\leq\bz\}.
\eeq
We will only examine the case $L+\lfloor \kappa_n s\rfloor\leq \lfloor ns_0\rfloor$ in detail since the other case works analogously.

\noindent The remainder of the proof is similar to the proof of Lemma \ref{Urate}. 
With $g(\bX_{i}):=(m_{(1)}(\bX_{i})-m_{(2)}(\bX_{i}))$ and $\bar f_n^{(s_0)}(\bX_i)=\frac{\sum_{j=\lfloor ns_0\rfloor+1}^n f_{j}(\bX_{i})}{\sjn f_{j}(\bX_{i})}$ it holds
\beq
&&\sup_{s\in[0,1]}\sup_{\bz\in\er^d}\left|\sum_{i=L+1}^{L+\lfloor \kappa_ns\rfloor}g(\bX_i)I\{|g(\bX_i)|>\kappa_n^{\frac 1{r}}\}\bar f_n^{(s_0)}(\bX_i)\omega_n(\bX_i)I\{\bX_i\leq\bz\}\right|\\
&\leq&\sum_{i=L+1}^{L+\kappa_n} |g(\bX_i)|I\{|g(\bX_i)|>\kappa_n^{\frac 1{r}}\}
\eeq
and further
\beq
P\left( \sum_{i=L+1}^{L+\kappa_n}|g(\bX_i)|I\{|g(\bX_i)|>\kappa_n^{\frac 1{r}}\}>C\kappa_n\right)&\leq&C^{-1}\kappa_n^{-1}C_m\kappa_n\kappa_n^{\frac1r-1}
\eeq
by the Markov inequality with
\beq
E\left[|g(\bX_i)|I\{|g(\bX_i)|>\kappa_n^{\frac 1{r}}\}\right]&=&E\left[|g(\bX_i)|^r|g(\bX_i)|^{-(r-1)}I\{|g(\bX_i)|>\kappa_n^{\frac 1{r}}\}\right]\\
&\leq&\kappa_n^{-\frac{r-1}r}E[|g(\bX_i)|^{r}]\\
&\leq&C_m\kappa_n^{\frac1r-1}
\eeq
for all $i$ and for some $C_m<\infty$ by assumption (M).
Thus we can rewrite our assertion as
\[P\left(\sup_{s\in[0,1]}\sup_{\varphi\in\cf_n}\left|\left(\sum_{i=L+1}^{L+\lfloor \kappa_ns\rfloor} \varphi(\bX_{i})-\int \varphi dP\right)\right|>C\kappa_n\right)\leq \bar C\kappa_n^{\frac1{r}-1},\]
with the function class
\[\cf_n:=\left\{\bx\mapsto g(\bx)I\{|g(\bx)|\leq \kappa_n^{\frac 1{r}}\}\bar f_n^{(s_0)}(\bx)\omega_n(\bx)I\{\bx\leq\bz\}:\bz \in \er^d\right\}.\]

\noindent To replace the supremum over $\varphi$ by a maximum we cover $\cf_n$ by finitely many brackets $[\varphi_\bj^l,\varphi_\bj^u]_{\bj\in\times_{i=1}^d\{1,\ldots,N_i\}}$
where 
\beq
\varphi_{\bj}^u(\bx)&:=&g(\bx)I\{|g(\bx)|\leq \kappa_n^{\frac 1r}\}I\{g(\bx)\geq 0\}\bar f_n^{(s_0)}(\bx)\omega_n(\bx)I\{\bx\leq\bz_{\bj}\}\\
&&+g(\bx)I\{|g(\bx)|\leq \kappa_n^{\frac 1r}\}I\{g(\bx)<0\}\bar f_n^{(s_0)}(\bx)\omega_n(\bx)I\{\bx\leq\bz_{\bm{j-1}}\}\eeq
and
\beq
\varphi_{\bj}^u(\bx)&:=&g(\bx)I\{|g(\bx)|\leq \kappa_n^{\frac 1r}\}I\{g(\bx)\geq 0\}\bar f_n^{(s_0)}(\bx)\omega_n(\bx)I\{\bx\leq\bz_{\bm{j-1}}\}\\
&&+g(\bx)I\{|g(\bx)|\leq \kappa_n^{\frac 1r}\}I\{g(\bx)<0\}\bar f_n^{(s_0)}(\bx)\omega_n(\bx)I\{\bx\leq\bz_{\bj}\}\eeq
and $\bj$, $\bz_\bj$ are defined as in the proof of Lemma \ref{Urate}.
The total number of brackets $J_n:=N_{[\ ]}(\epsilon_n,\cf_n,\|\cdot\|_{L_1(P)})$ needed to cover $\cf_n$ is again of order
$J_n=O(\epsilon_n^{-d})$, which follows analogously to but easier than the proof of Lemma A.7 in \cite{Mohr2018}.
Now we proceed completely analogously to the proof of Lemma \ref{Urate} by replacing the supremum over $s$ by a maximum as well and applying Liebscher's Theorem. Since the arguments are the same as in the aforementioned proof we omit this part for the sake of brevity.
\end{proof}


\begin{proof}[Proof of Lemma \ref{munif}]
It holds
\beq
&&P\left(\sup_{s\in[0,1]}\sup_{z\in\er^d}\left| \sum_{i=L+1}^{L+\lfloor\kappa_ns\rfloor} (\bar m_n(\bX_{i})-\hat m_n(\bX_{i}))\omega_n(\bX_{i})I\{\bX_{i}\leq\bz\}\right|>C\kappa_n\right)\\
&\leq&P\left(\sum_{i=L+1}^{L+\kappa_n} |\bar m_n(\bX_{i})-\hat m_n(\bX_{i})|\omega_n(\bX_{i})>C\kappa_n\right)\\
&\leq&P\left(\sup_{\bm{x}\in\bm{J}_n}|\bar{m}_n(\bm{x})-\hat{m}_n(\bm{x})|>C\right)\\
&\leq& C^{-1}E[\sup_{\bm{x}\in\bm{J}_n}|\bar{m}_n(\bm{x})-\hat{m}_n(\bm{x})|]
\eeq
by the Markov inequality. Further by Lemma \ref{mdachRate} with assumption (B) it holds that
\[\frac{\sup_{\bm{x}\in\bm{J}_n}|\bar{m}_n(\bm{x})-\hat{m}_n(\bm{x})|}{n^{-\zeta}}\xrightarrow[n\to\infty]{P}0\]
which implies
\[\frac{E[\sup_{\bm{x}\in\bm{J}_n}|\bar{m}_n(\bm{x})-\hat{m}_n(\bm{x})|]}{n^{-\zeta}}\xrightarrow[n\to\infty]{}0\]
and thus for sufficiently large $n$
\beq
E[\sup_{\bm{x}\in\bm{J}_n}|\bar{m}_n(\bm{x})-\hat{m}_n(\bm{x})|]&\leq& n^{-\zeta}\\
&\leq& \kappa_n^{-\zeta}
\eeq
for $\kappa_n\leq n$. This completes the proof.
\end{proof}

\bibliographystyle{apa}
\bibliography{mybibfile}

\end{document}